\documentclass[12pt,3p]{elsarticle}

\usepackage{graphicx}
\usepackage{amssymb}
\usepackage{amsthm}
\usepackage{lineno}
\usepackage{amssymb}
\usepackage{amsthm}
\usepackage{color}
\usepackage[utf8]{inputenc}
\usepackage{fixltx2e,fix-cm}
\usepackage{amsmath}
\usepackage{amsfonts}
\usepackage{subfigure}
\usepackage{color}
\usepackage{enumerate}
\usepackage{lineno}
\usepackage{graphicx}
\usepackage{makeidx}
\usepackage{multicol}
\usepackage{lipsum}
\usepackage{graphicx}
\usepackage{verbatim}
\usepackage{comment}
\usepackage{tabularx}
\usepackage{tikz}
\usetikzlibrary{arrows}
\usetikzlibrary{decorations.pathreplacing}

\usetikzlibrary{positioning}
\tikzset{main node/.style={circle,fill=white,draw,minimum size=1cm,inner sep=0pt},}

\newtheorem{theorem}{Theorem}[section] 
\newtheorem{proposition}[theorem]{Proposition}
\newtheorem{lemma}[theorem]{Lemma}

\theoremstyle{definition}
\newtheorem{definition}[theorem]{Definition}

\theoremstyle{remark}
\newtheorem{remark}[theorem]{Remark}
\newtheorem{example}[theorem]{Example}

\newcolumntype{L}[1]{>{\raggedright\let\newline\\\arraybackslash\hspace{0pt}}m{#1}}

\journal{Involve}

\begin{document}

\begin{frontmatter}

\title{Tile Based Modeling of DNA Self-Assembly for Two Graph Families with Appended Paths}

\author[label1]{Chloe Griffin}
\ead{dcgriffin001@gmail.com}
\address[label1]{Converse University, 580 E. Main Street, Spartanburg, South Carolina 29302}
\author[label1]{Jessica Sorrells}
\ead{jessica.sorrells@converse.edu}

\begin{abstract}
 Branched molecules of deoxyribonucleic acid (DNA) can self-assemble into nanostructures through complementary cohesive strand base pairing. The production of DNA nanostructures is valuable in targeted drug delivery and biomolecular computing. With theoretical efficiency of laboratory processes in mind, we use a flexible tile model for DNA assembly. We aim to minimize the number of different types of branched junction molecules necessary to assemble certain target structures. We represent target structures as discrete graphs and branched DNA molecules as vertices with half-edges. We present the minimum numbers of required branched molecule and cohesive-end types under three levels of restrictive conditions for the tadpole and lollipop graph families. These families represent cycle and complete graphs with a path appended via a single cut-vertex. We include three general lemmas regarding such vertex-induced path subgraphs. Through proofs and examples, we demonstrate the challenges that can arise in determining optimal construction strategies. 
\end{abstract}


\end{frontmatter}

\section{Introduction}

In 1974, Professor Norio Taniguchi of the Tokyo Science University coined the term “nanotechnology” to describe the engineering of materials at an atomic level \cite{Taniguchi}. Nadrian Seeman, a dedicated crystallographer, invented the field of DNA nanotechnology in an attempt to better the crystallization process \cite{Nadrian1}. Since that time, the use of nanotechnology has expanded to a wide variety of fields, including  medicine, electronics, food, fuel cells, solar cells, and batteries \cite{NanotechUses}. In the early 2000s, DNA nanostructure self-assembly became a key technology for targeted drug delivery, biosensors, and biomolecular computing \cite{labean2007constructing} \cite{seeman2007overview}. The need for nanostructures have prompted researchers to purposefully guide the DNA self-assembly process. \par
Specially designed DNA strands can bond together to form geometric structures. The characteristics of a target structure are dependent on its applications. Our goal is to theoretically minimize the number of components necessary to obtain target structures. Mathematically, one can represent final structures as discrete graphs \cite{mintiles} \cite{ellis2019tile}. Incorporating graph theory has resulted in new design strategies for the self-assembly process, but has also introduced new graph invariants and combinatorial questions.

\subsection{Flexible Tile Based Model}

By representing a target structure with a discrete graph, we are able to model and create design strategies for optimal design of component DNA building blocks. We follow the flexible tile model described by Ellis-Monaghan et al. in \cite{mintiles}.  A $k$\textit{-armed branched junction molecule} is a molecule of DNA consisting of a center point with extending arms, as shown in Figure \ref{fig:branchedjunction}. In this model, $k$-armed branched junction molecules are used as the building blocks of DNA structures. The extended strands of DNA are capable of bonding with any other extended strand with a complementary sequence of Watson-Crick bases. These bonding strands are referred to as \textit{cohesive-ends}. In our model, a vertex represents the center of a branched junction molecule, and half-edges represent cohesive ends. Two half-edges may bond to form a complete edge between two vertices. 

\begin{figure}[h]
 \centering \includegraphics[width = 5 cm]{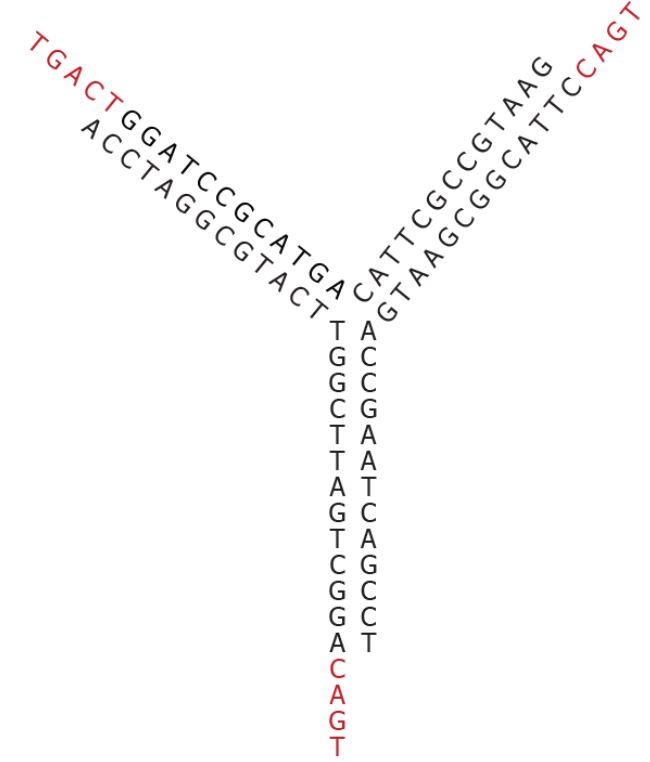} 
\caption{A $k$-armed branched junction molecule with cohesive-ends} \label{fig:branchedjunction}
\end{figure}

\begin{definition} The abstract representation of a $k$-armed branched junction molecule is called a \textit{tile}. Each tile consists of a vertex with $k$ extending half-edges. \end{definition}

Half-edges are labeled with hatted and un-hatted letters. These letters are referred to as \textit{bond-edge types}, and they represent specific sequences of DNA. The hatted and un-hatted version of a letter are complementary cohesive-ends that can bond together. Formally, we represent tiles as sets of hatted and un-hatted letters. The flexible tile model assumes tile arms are sufficiently long and flexible in order to bond to arms of other tiles.
 
 \begin{definition} A collection of tiles is called a \textit{pot}.
 \end{definition} 
 
 We write a pot of tiles as a set of tiles, as illustrated by Figure \ref{fig:potoftiles}. \\
 
  \begin{figure}[h]
            \centering \includegraphics[width = 7 cm]{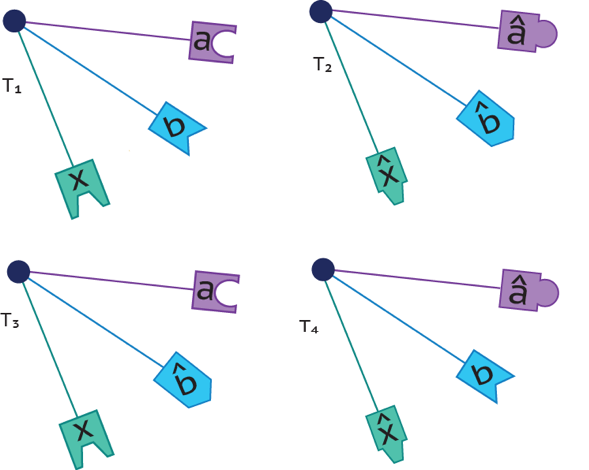} 
            \caption{A pot of four tiles: $\{\{a,b,x\}, \{\hat{a},\hat{b},\hat{x}\}, \{a, \hat{b}, x\}, \{\hat{a},b,\hat{x}\}\}$} \label{fig:potoftiles}
\end{figure}
 
 We refer to a collection of tiles joined together as a \textit{complex}; a complex with no unmatched half-edges is a \textit{complete} complex. Full edges in a complex are referred to as \textit{bond-edges}. A collection of tiles joined together in a complete complex is viewed as a graph $G$ representing a DNA nanostructure. If a graph $G$ can be constructed as a complete complex from a given pot $P$, we say that $P$ \textit{realizes} $G$. Figure \ref{fig:Pot} illustrates a graph realized by a collection of tiles.
 
  \begin{figure}[h]
            \centering \includegraphics[width = 5 cm]{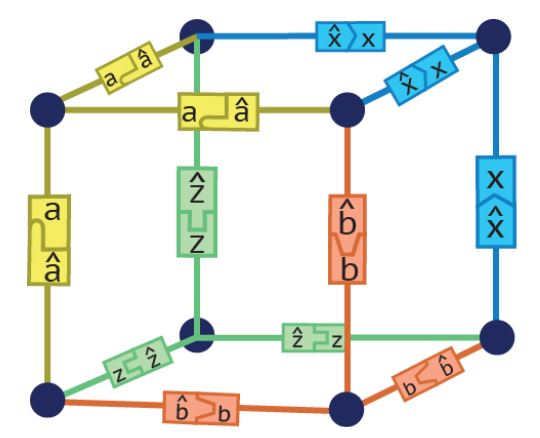} 
            \caption{Cube graph with half-edge labels} \label{fig:Pot}
\end{figure}
 
In a laboratory setting, an attempt to produce a target structure can result in an excess of costly branched junction molecules. Mathematical research done within the flexible tile model seeks to find a pot with the minimum numbers of bond-edge types and tile types required to construct a given target complex \cite{mintiles} \cite{dnamotivation}. We aim to find an accurate edge labeling of a specific graph with as few different half-edge labels as possible. We consider the experimental potential for a pot to realize graphs other than the target. Therefore, we determine the minimum number of bond-edge and tile types under three levels of restrictive conditions given in \cite{mintiles}: 

\begin{itemize}     
\item \textbf{Scenario 1.} Complete complexes of smaller size (that is, graphs with fewer vertices) than the target complex are allowed to be realized by the pot.
\item \textbf{Scenario 2.} Complete complexes of the same size as (that is, graphs with the same number of vertices), but not isomorphic to, the target complex are allowed to be realized by the pot. No smaller complete complexes are allowed to be realized.
\item \textbf{Scenario 3.} No complete complexes of smaller size and no non-isomorphic complexes of the same size are allowed to be realized by the pot; the target graph is the only graph of that order realized by the pot.
\end{itemize}

We use \textit{order} to mean the number of vertices in a graph; we use \textit{size} to refer to the number of vertices in an incomplete complex. With the flexible tile model, the minimum numbers of tile and bond-edge types in all three scenarios have been found for trees, cycles, complete, complete bipartite, wheel, windmill, and gear graphs \cite{mintiles} \cite{Mattamira} \cite{WheelGraphs}. In this work, we prove a collection of results in all scenarios for lollipop and tadpole graphs of all orders. The majority our results in Scenarios 1 and 3 are exact values; when an exact value is not determined, we provide two consecutive integers as bounds. Scenario 2 presents the greatest challenge. Still, in many cases we are able to give exact values. For certain tadpole and lollipop graphs in Scenario 2, we narrow the range for optimal construction by providing both lower and upper bounds. \par
 Currently, there is no proven theory in the area of flexible tile based DNA self-assembly that defines the relationship between graphs and subgraphs. We present the following definitions to aid in explaining our work. 
 \begin{definition} A \textit{cut-vertex} of a graph is a vertex whose deletion increases the number of components. \cite{WestGraphTheory} \end{definition}
 \begin{definition} A \textit{vertex-induced subgraph} is a subset of the vertices of a graph together with any edges whose endpoints are both in this subset. \cite{vertexinduced} \end{definition}
 Our work illustrates the differences between complete graphs and lollipop graphs, and between cycle graphs and tadpole graphs; these serve as case studies of certain vertex-induced subgraphs joined to the remainder of the graph via a single cut-vertex. Our results give insight into general rules for self-assembly when appending a path onto a graph. For example, in Section \ref{sec:prelim} we give conditions in which a tile or bond-edge type can be used twice within an appended path component. By exploring lollipop and tadpole graphs, we also gain perspective on the difficulties of determining optimal values for graph families that expand in order in two distinct ways.

\subsection{Methods and Notation}

We follow the notation of \cite{mintiles} to denote the minimum number of tiles needed to construct a target graph, $G$. We use the notation $T_{i}(G)$ for $i = 1,2,3$ where $i$ value corresponds to the scenario in which the pot is being described. Similarly, $B_i(G)$ denotes the minimum number of bond-edge types needed. $T_i(G)$ and $B_i(G)$ are new graph invariants. To notate the number of distinct even vertex degrees, distinct odd vertex degrees, and total distinct vertex degrees that appear in a graph $G$, we use $ev(G)$, $ov(G)$, and $av(G)$, respectively. In order to more efficiently illustrate labeled graph edges, we use colors for different bond-edge types and arrows oriented toward the hatted version of the bond-edge type, as illustrated in Figure \ref{fig:arrowexplain}. We say that bond-edges have the same orientation if the two edges, represented as arrows, point in the same direction under consideration in a given proof. We denote different tile types as $t_1, t_2,$ etc.

\begin{figure}[h!]
     \centering
	 \begin{tikzpicture}[transform shape, scale = 0.7]
			
 \node[main node] (a) at (0,0) {$t_1$};
 \node[main node] (b) at (0,4) {$t_1$};
 \node[main node](c) at (3,2) {$t_2$};
 \node[main node] (d) at (5,2) {$t_1$};		 
		 
\path[draw,thick]
(a) edge node [near start,left]{$a$} (b)
(a) edge node [near start,right]{$\hat{a}$} (c)
(b) edge node [near start,left]{$\hat{a}$} (a)
(b) edge node [near start,above]{$a$} (c)
(c) edge node [near start,right]{$a$} (a)
(c) edge node [near start,above]{$\hat{a}$} (b)
(c) edge node [near start,above]{$b$} (d)
(d) edge node [near start,above]{$\hat{b}$} (c);

 \node[main node] (e) at (9,0) {$t_1$};
 \node[main node] (f) at (9,4) {$t_1$};
 \node[main node](g) at (12,2) {$t_2$};
 \node[main node] (h) at (14,2) {$t_1$};

\path[draw,thick,color=red,->]
(e) edge node []{} (f)
(g) edge node []{} (e)
(f) edge node []{} (g);

\path[draw,thick,color=blue,->]
(g) edge node []{} (h);

\end{tikzpicture}

\caption{Colored arrow edge labeling of a graph} 
\label{fig:arrowexplain}
\end{figure}
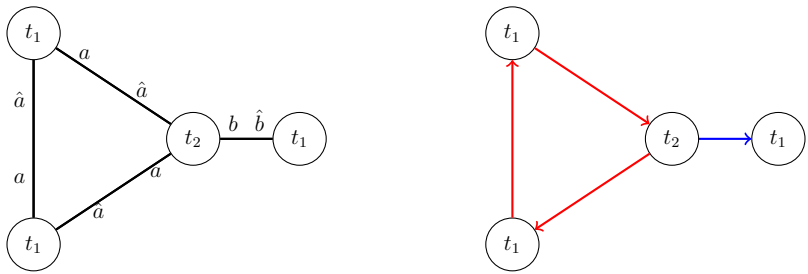

In order to determine minimum values in Scenario 1, we apply Corollary 1 and Theorem 1 of \cite{mintiles}, which provide that $B_1(G)=1$ and $av(G) \leq T_1(G) \leq ev(G) + 2ov(G)$ for all graphs $G$. 
For Scenario 2, it is useful to note from Theorem 2 of \cite{mintiles} that $B_2(G)+1 \leq T_2(G)$ for all graphs $G$. We also utilize the following system of linear equations and corresponding augmented matrix introduced in \cite{mintiles}.

\begin{definition}\label{matrix} Let $P = \{ t_1,...,t_p \}$ be a pot and let $z_{i,j}$ denote the net number of cohesive-ends of type $a_i$ on tile $t_j$, where un-hatted cohesive-ends are counted positively and hatted cohesive-ends are counted negatively. Then the following system of equations must be satisfied by any complete graph constructed from $P$:
\begin{eqnarray*} z_{1,1}r_1+z_{1,2}r_2+...+z_{1,p}r_p &=& 0 \\
& \vdots & \\
z_{m,1}r_1+z_{m,2}r_2+...+z_{m,p}r_p &=& 0 \\
r_1 + r_2 + ... + r_p &=& 1 \end{eqnarray*}

\noindent The \emph{construction matrix} of $P$, denoted $M(P)$, is the corresponding augmented matrix:

\begin{equation*}
M(P) = \begin{bmatrix} \begin{array}{*{20}{cccc|c}}
   {{z_{1,1}}} & {{z_{1,2}}} &  \ldots  & {{z_{1,p}}} & 0  \\
    \vdots  &  \vdots  & {\ddots} &  \vdots  & {} \vdots \\
   {{z_{m,1}}} & {{z_{m,2}}} &  \ldots  & {{z_{m,p}}} & 0  \\
1 & 1 &  \ldots  & 1 & 1  \\
 \end{array} \end{bmatrix} 
\end{equation*}
\end{definition}

In Scenario 2, the construction matrix can sometimes be used to determine if any graphs smaller than the target graph can be realized from a pot. $M(P)$ has solutions $\langle r_1,...,r_p \rangle$, where, for the purposes of this model, components $r_i$ of the vector solutions are tile proportions, and $r_i \in \mathbb{Q}^+$. The least common denominator of the $r_i's$ is the order of the smallest graph that can be realized by $P$ \cite{mintiles}. 

\begin{definition} The solution space of the construction matrix of a pot $P$ is called the {\it spectrum} of $P$ and is denoted  $\mathcal S(P)$. \cite{jonoska} \end{definition}

If $\mathcal{S}(P)$ consists of a unique solution, conclusions are often straightforward. However, in Sections \ref{Lollipopgraphsection} and \ref{Tadpolegraphsection} the spectrums of provided pots frequently have one or more degrees of freedom. To overcome this difficulty, we use alternative methods to those found in \cite{mintiles}. This challenge, among others in Scenarios 2 and 3, can be witnessed in many of the proofs for lollipop and tadpole graphs. In general, determining optimal design strategies in the flexible tile model for Scenario 2 is NP-complete; a full discussion of the computational complexity of finding satisfactory pots can be found in \cite{Sorrells}. This validates the pragmatic methods seen in our proofs. We continue to utilize the construction matrix and spectrum of a given pot and examine all possibilities for tile proportions that remain valid under the restrictions of the flexible tile model (i.e. non-negative values less than or equal to one). This process restricts the values of free variables present in those tile proportions. We exhaust all possibilities for how the arms of the tiles, in the proportions determined to be valid, can bond to one another; this typically results in a proven lower bound for the size of any complete complex realized by the pot.

In Scenario 3, general analysis of graph isomorphism is the primary technique, along with continued verification via the spectrum of a pot that no smaller graphs can be realized; our work in Scenario 3 closely mirrors the strategies found in \cite{mintiles}. Finally, it is helpful to note from Proposition 1 of \cite{mintiles} that $B_3(G) \geq B_2(G) \geq B_1(G)$ and $T_3(G) \geq T_2(G) \geq T_1(G)$.

\section{General Results for Appended Paths}\label{sec:prelim}
The following lemmas are useful in several proofs that follow in Sections \ref{Lollipopgraphsection} and \ref{Tadpolegraphsection}. These are general results that can be applied in Scenarios 2 or 3 to any graph with a vertex-induced path subgraph connected to the remainder of the graph via a single cut-vertex. \par

Lemmas \ref{S2NoRepeatedBondEdge} and \ref{S2RepeatedBondEdge} are reminders of the nuance required when determining optimal pots in Scenario 2.

\begin{lemma} \label{S2NoRepeatedBondEdge}
Let $G$ be a graph of order $m+n$ where $n \leq m$ and in which a vertex-induced path subgraph of $n$ vertices is connected to the remaining $m$ vertices of the graph via a single cut-vertex. In order to satisfy the requirements of Scenario 2, a single bond-edge type may not be repeated in the path subgraph. Thus, $B_2(G) \geq n$ and $T_2(G) \geq n$.
\end{lemma}

\begin{proof} Consider a graph $G_{m,n}$ which is created by appending a path of order $n$ to a graph $H$ of order $m$ via a single cut-vertex. Figure \ref{fig:Lemmafig1} illustrates a general form of such a graph $G_{m,n}$; note that $H$ may be any graph consisting of $m$ vertices, including the aforementioned cut-vertex. Suppose bond-edge type $a$ appears twice along the path. Let the number of vertices in the path between $H$ and the first instance of the repeated bond-edge type be denoted as $y$ (including the vertex with the first half-edge labeled with bond-edge type $a$, but not the vertex with the complementary half-edge),  the number of vertices between the two repeated bond-edge types denoted as $k$, and the number of vertices in the path from the second instance of the repeated bond-edge to the last vertex of the path denoted as $l$, as shown in Figure \ref{fig:Lemmafig1}. Note that $n=y+k+l$. \par 
If the edges with bond-edge type $a$ have the same orientation, the resulting pot realizes two graphs smaller than $G_{m,n}$, one of which is $H$ with a shorter path appended formed by removing the middle $k$ vertices, and the other is a cycle formed from the middle $k$ vertices (a single vertex with a loop edge if $k=1$). Both possibilities are shown in Figure \ref{fig:Lemmafig1} and are clearly smaller than $G_{m,n}$.

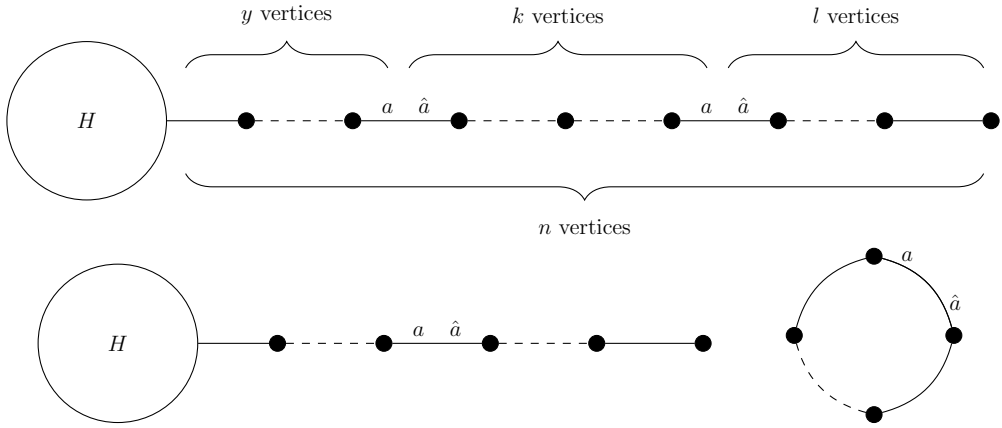
\begin{figure}[h!]
     \centering
	 \begin{tikzpicture}[transform shape, scale = 0.7]
 \node[main node, minimum size = 3cm] (a) at (1,0) {$H$};
 \node[main node, fill=black, minimum size = 0.3cm] (b) at (4,0) {};
 \node[main node, fill=black, minimum size = 0.3cm] (c) at (6,0) {};
 \node[main node, fill=black, minimum size = 0.3cm] (d) at (8,0) {};
 \node[main node, fill=black, minimum size = 0.3cm] (e) at (10,0) {};
 \node[main node, fill=black, minimum size = 0.3cm] (f) at (12,0) {};
 \node[main node, fill=black, minimum size = 0.3cm] (g) at (14,0) {};
 \node[main node, fill=black, minimum size = 0.3cm] (h) at (16,0) {};
 \node[main node, fill=black, minimum size = 0.3cm] (i) at (18,0) {};
		 
\path[draw]
(a) edge node []{} (b)
(c) edge node [above]{$a$ \hspace{2mm} $\hat{a}$} (d)
(f) edge node [above]{$a$ \hspace{2mm} $\hat{a}$} (g)
(h) edge node []{} (i);

\path[draw, dashed]
(b) edge node []{} (c)
(d) edge node []{} (e)
(e) edge node []{} (f)
(g) edge node []{} (h);

\draw [decorate,decoration={brace,amplitude=10pt},xshift=-4pt,yshift=0pt]
(3,1) -- (6.8,1) node [black,midway,xshift=0, yshift=1cm] 
{$y$ vertices};

\draw [decorate,decoration={brace,amplitude=10pt},xshift=-4pt,yshift=0pt]
(7.2,1) -- (12.8,1) node [black,midway,xshift=0, yshift=1cm] 
{$k$ vertices};

\draw [decorate,decoration={brace,amplitude=10pt},xshift=-4pt,yshift=0pt]
(13.2,1) -- (18,1) node [black,midway,xshift=0, yshift=1cm] 
{$l$ vertices};

\draw [decorate,decoration={brace,amplitude=10pt,mirror},xshift=-4pt,yshift=0pt]
(3,-1) -- (18,-1) node [black,midway,xshift=0, yshift=-1cm] 
{$n$ vertices};
\end{tikzpicture}

 \begin{tikzpicture}[transform shape, scale = 0.7]
 \node[main node, minimum size = 3cm] (a) at (1,0) {$H$};
 \node[main node, fill=black, minimum size = 0.3cm] (b) at (4,0) {};
 \node[main node, fill=black, minimum size = 0.3cm] (c) at (6,0) {};
 \node[main node, fill=black, minimum size = 0.3cm] (d) at (8,0) {};
 \node[main node, fill=black, minimum size = 0.3cm] (e) at (10,0) {};
 \node[main node, fill=black, minimum size = 0.3cm] (f) at (12,0) {};
		 
\path[draw]
(a) edge node []{} (b)
(c) edge node [above]{$a$ \hspace{2mm} $\hat{a}$} (d)
(e) edge node []{} (f);

\path[draw, dashed]
(b) edge node []{} (c)
(d) edge node []{} (e);
\end{tikzpicture} \hspace{7mm}
\begin{tikzpicture}[transform shape, scale = 0.7]
 \node[main node, fill=black, minimum size = 0.3cm] (a) at (-1.5,0) {};
 \node[main node, fill=black, minimum size = 0.3cm] (b) at (0,1.5) {};
  \node[main node, fill=black, minimum size = 0.3cm] (c) at (1.5,0) {};
 \node[main node, fill=black, minimum size = 0.3cm] (d) at (0,-1.5) {};
		 
\path[draw]
(a) edge [bend left] node []{} (b)
(b) edge [bend left] node [near start,above]{$a$} (c)
(b) edge [bend left] node [near end,right]{$\hat{a}$} (c)
(c) edge [bend left] node []{} (d);

\path[draw, dashed]
(d) edge [bend left] node []{} (a);
\end{tikzpicture}

\caption{$G_{m,n}$ and two smaller graphs realized in proof of Lemma \ref{S2NoRepeatedBondEdge}} 
\label{fig:Lemmafig1}
\end{figure}

If the edges with bond-edge type $a$ have opposite orientation, the resulting pot realizes at least two graphs different from $G_{m,n}$. One such graph, denoted as $G_1$, is the combination of the $k$ vertices in the middle of the path together with two copies of the $l$ vertices at the end of the path. $G_1$ has two half-edges labeled with bond-edge type $a$ on either side of the middle $k$ vertices. Each half-edge bonds with a copy of the $l$ vertices. $G_1$ is a path graph of order $2l +k$. Another such graph, denoted as $G_2$, is the combination of two copies of $H$ together with the first $y$ vertices in the path, joined by the middle $k$ vertices in the path. $G_2$ has two half-edges labeled with bond-edge type $a$ on either side of the middle $k$ vertices. Each half-edge bonds to a copy of the complex consisting of $H$ and the first $y$ vertices in the path. The order of $G_2$ is $2(m+y) + k$. These two graphs are illustrated in Figure \ref{fig:Lemmafig2}.  We claim at least one of these two potential graphs must be smaller than $G_{m,n}$.  \par

\begin{figure}[h!]
     \centering
     
     \begin{tikzpicture}[transform shape, scale = 0.7]
 \node[main node, fill=black, minimum size = 0.3cm] (a) at (2,0) {};
 \node[main node, fill=black, minimum size = 0.3cm] (b) at (4,0) {};
 \node[main node, fill=black, minimum size = 0.3cm] (c) at (6,0) {};
 \node[main node, fill=black, minimum size = 0.3cm] (d) at (8,0) {};
 \node[main node, fill=black, minimum size = 0.3cm] (e) at (10,0) {};
 \node[main node, fill=black, minimum size = 0.3cm] (f) at (12,0) {};
 \node[main node, fill=black, minimum size = 0.3cm] (g) at (14,0) {};
 \node[main node, fill=black, minimum size = 0.3cm] (h) at (16,0) {};
 \node[main node, fill=black, minimum size = 0.3cm] (i) at (18,0) {};
		 
\path[draw]
(a) edge node []{} (b)
(c) edge node [above]{$a$ \hspace{2mm} $\hat{a}$} (d)
(f) edge node [above]{$\hat{a}$ \hspace{2mm} $a$} (g)
(h) edge node []{} (i);

\path[draw, dashed]
(b) edge node []{} (c)
(d) edge node []{} (e)
(e) edge node []{} (f)
(g) edge node []{} (h);

\draw [decorate,decoration={brace,amplitude=10pt},xshift=-4pt,yshift=0pt]
(2,1) -- (6.5,1) node [black,midway,xshift=0, yshift=1cm] 
{$l$ vertices};

\draw [decorate,decoration={brace,amplitude=10pt},xshift=-4pt,yshift=0pt]
(7.5,1) -- (12.5,1) node [black,midway,xshift=0, yshift=1cm] 
{$k$ vertices};

\draw [decorate,decoration={brace,amplitude=10pt},xshift=-4pt,yshift=0pt]
(13.5,1) -- (18,1) node [black,midway,xshift=0, yshift=1cm] 
{$l$ vertices};
\end{tikzpicture}

\vspace{1cm}
	 \begin{tikzpicture}[transform shape, scale = 0.7]
 \node[main node, minimum size = 3cm] (i) at (1,-2) {$H$};
 \node[main node, minimum size = 3cm] (a) at (1,2) {$H$};
 \node[main node, fill=black, minimum size = 0.3cm] (b) at (4,2) {};
 \node[main node, fill=black, minimum size = 0.3cm] (c) at (6,2) {};
 \node[main node, fill=black, minimum size = 0.3cm] (d) at (8,2) {};
 \node[main node, fill=black, minimum size = 0.3cm] (e) at (8,0) {};
 \node[main node, fill=black, minimum size = 0.3cm] (f) at (8,-2) {};
 \node[main node, fill=black, minimum size = 0.3cm] (g) at (6,-2) {};
 \node[main node, fill=black, minimum size = 0.3cm] (h) at (4,-2) {};
 \node[main node, fill=white, draw=white, minimum size = 0cm] (j) at (10.5,0) {$k$ vertices};

\path[draw]
(a) edge node []{} (b)
(c) edge node [above]{$a$ \hspace{2mm} $\hat{a}$} (d)
(f) edge node [above]{$a$ \hspace{2mm} $\hat{a}$} (g)
(h) edge node []{} (i);

\path[draw, dashed]
(b) edge node []{} (c)
(d) edge node []{} (e)
(e) edge node []{} (f)
(g) edge node []{} (h);

\draw [decorate,decoration={brace,amplitude=10pt},xshift=-4pt,yshift=0pt]
(3,3) -- (7,3) node [black,midway,xshift=0, yshift=1cm] 
{$y$ vertices};

\draw [decorate,decoration={brace,amplitude=10pt},xshift=-4pt,yshift=0pt]
(7,-3) -- (3,-3) node [black,midway,xshift=0, yshift=-1cm] 
{$y$ vertices};

\draw [decorate,decoration={brace,amplitude=10pt},xshift=-4pt,yshift=0pt]
(9,2.2) -- (9,-2.2) node [black,midway,xshift=10, yshift=0cm] {};
\end{tikzpicture}

\caption{Two smaller graphs $G_1$ and $G_2$ realized in proof of Lemma \ref{S2NoRepeatedBondEdge}} 
\label{fig:Lemmafig2}
\end{figure}
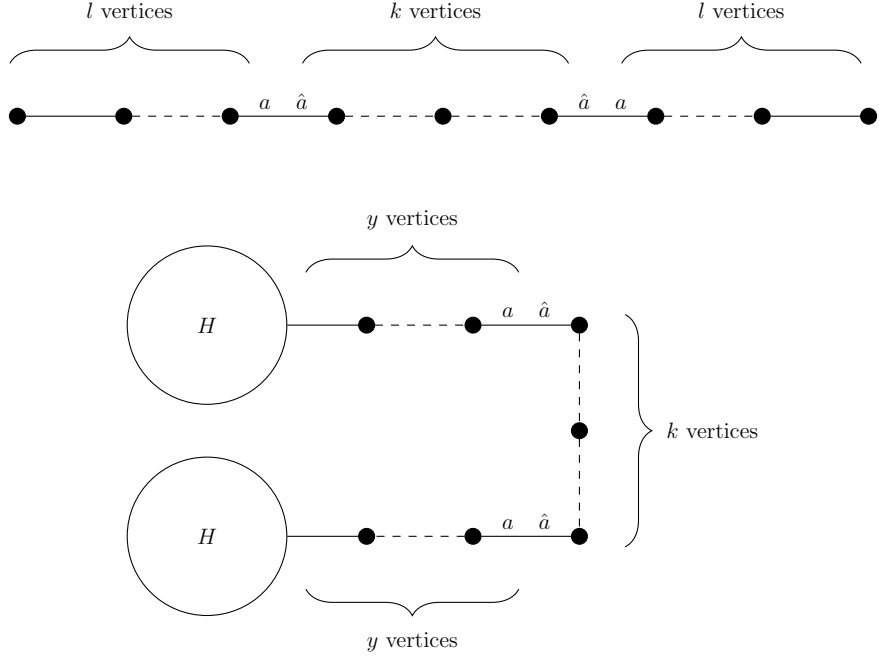

 We show that if the order of either $G_1$ or $G_2$ exceeds the order of $G_{m,n}$, then the other graph will be of smaller order than $G_{m,n}$. Suppose $G_1$ is larger than $G_{m,n}$, so $2l + k > m + n$. Substituting for $l$ gives $2(n-k-y)+k > m+n$. By rearranging, $2(m+y)+k<m+n$, showing that $G_2$ is smaller than $G_{m,n}$. Conversely, suppose $G_2$ is larger than $G_{m,n}$, so $2(y + m) + k> m + n$. Substituting for $y$ gives $m+2(n-k-l)+k>n$, so $2l+k < m+n$. Hence, $G_1$ must be smaller than $G_{m,n}$.\par

Suppose $G_1$ and $G_2$ are of equal order. Then $2l+k=2(m+y)+k$, so $l=m+y$. This would require $n \geq m+1$, and thus is impossible. Note, however, that this would imply $G_{m,n}$, $G_1$, and $G_2$ are all of equal order. \par

It follows that there can be no repeated bond-edge types or tile types within the appended path subgraph of $G_{m,n}$ if we are to prevent the realization of smaller graphs. Note that all tile types used to label vertices in the path will also be distinct if no bond-edge types can be repeated. Thus, in order to satisfy Scenario 2 there must be $n$ bond-edge types and $n$ tile types used to label the vertices in the path subgraph of the graph $G_{m,n}$. 
\end{proof}

\begin{lemma} \label{S2RepeatedBondEdge}
Let $G$ be a graph in which a vertex-induced path subgraph of $n$ vertices is connected to the remaining $m$ vertices of the graph via a single cut-vertex such that $n>m$. In order to satisfy the requirements of Scenario 2 a bond-edge type can be used to label at most two edges of the path subgraph, and these two edges must be among the $n-m+1$ edges closest to the other $m$ vertices of the graph. In addition, the edges must be labeled with opposite orientation. Thus, $B_2(G) \geq \lceil \frac{n-m+1}{2} \rceil + m-1$ and $T_2(G) \geq \lceil \frac{n-m+1}{2} \rceil +m$.
\end{lemma}

\begin{proof}
Suppose bond-edge type $a$ is used to label two edges in the appended path of $G$. The proof of Lemma \ref{S2NoRepeatedBondEdge} shows that in order to avoid the realization of a graph strictly smaller than $G$ it is necessary that $n > m,$ and the edges labeled with bond-edge type $a$ are labeled with opposite orientation and spaced along the path such that $l=m+y$, as shown in Figure \ref{fig:Lemmafig1}. Thus, $l \geq m$ and both occurrences of bond-edge type $a$ must occur on the $n-m+1$ edges of the path closest to the $m$ vertices not in the path. That is, they must not be among the $m-1$ edges of the path farthest from $H$ as shown in Figure \ref{fig:Lemmafig1}. Note that a bond-edge type cannot be used to label more than two edges of the path, since this would necessarily result in two edges with matching orientation. This implies at most $\lfloor \frac{n-m+1}{2} \rfloor$ bond-edge types can be used twice and at least $\lceil \frac{n-m+1}{2}\rceil$ bond-edge types are needed to label the $n-m+1$ edges of the path closest to the $m$ vertices not in the path, resulting in a minimum of $\lceil \frac{n-m+1}{2}\rceil + m -1$ bond-edge types to label the entire path. Since labeling two vertices of the path with the same tile type repeats two bond-edge types, a minimum of $\lceil \frac{n-m+1}{2}\rceil + m$ tile types must be used to label the vertices of the path.
\end{proof}

The final lemma of this section addresses Scenario 3, in which limitations are more severe, and thus the comparative length of the path is less relevant.

\begin{lemma} \label{S3NoRepeatedBondEdge}
Let $G$ be a graph that is not a path and consists of a vertex-induced subgraph $H$ of order $m$ connected to a vertex-induced path subgraph of order $n$ via a single cut-vertex. In order to satisfy the requirements of Scenario 3, a distinct bond-edge type must be used to label each edge of the path subgraph. Thus, $B_3(G) \geq n,$ and $T_3(G) \geq n$.
\end{lemma}

\begin{proof} The proof of Lemma \ref{S2NoRepeatedBondEdge} applies even when $n > m$ to show the realization of a smaller or non-isomorphic graph when a bond-edge type is used to label more than one edge of the path subgraph. Note that the only possibility for a graph realized in the proof to be isomorphic to $G$ is if $H$ is the path graph $P_m$. $H$ is necessarily not $P_m$ since if it were then $G$ would be $P_{m+n}$. 
\end{proof}

\begin{remark} $B_i(G)$ and $T_i(G)$ are known for all path graphs $P_n$ since path graphs are tree graphs \cite{mintiles}. Lemma \ref{S3NoRepeatedBondEdge} addresses all non-path graphs with appended path subgraphs. \end{remark}

Note that Lemmas \ref{S2NoRepeatedBondEdge}, \ref{S2RepeatedBondEdge}, and \ref{S3NoRepeatedBondEdge} only address repetition of bond-types within a path subgraph of a graph appended to the remainder of the graph via a single cut-vertex. These lemmas do not provide information about when bond-types used in the path can be repeated on other edges in the graph. Our results illustrate the differences that arise between two different types of non-path subgraphs with regard to ability to repeat bond-edge types between the path subgraph and the remainder of the graph while satisfying the restrictive conditions.\par

\section{Lollipop Graphs}\label{Lollipopgraphsection}

\begin{definition} A \textit{lollipop graph} is a complete graph $K_m$ connected to a path $P_n$ through a single bridging vertex of degree $m$. \cite{LollipopCite}
\end{definition} 

We denote a lollipop graph as $L_{m,n},$ where $m$ is the number of vertices in the complete graph subgraph and $n$ is the number of vertices in the extending path. Thus, the order of $L_{m,n}$ is $m + n$. We will often refer to the degree $m$ vertex as the \textit{bridging vertex}. Note that the bridging vertex is a cut-vertex and the path is a vertex-induced subgraph. \par 

$B_i(K_m)$ and $T_i(K_m)$ for $i=1,2,3$ are known and these values can be found in Table \ref{table:Complete Graph}. The complete graph is a vertex-induced subgraph of $L_{m,n}$ and, of course, is a maximum clique (an induced subgraph that is a complete graph \cite{vertexinduced}) of the graph. We will occasionally refer to the values of $T_i(K_m)$ and $B_i(K_m)$ in our proofs for $L_{m,n}$, as some of the same labeling strategies and lower bounds still apply. 

\renewcommand*{\arraystretch}{1.2}
\begin{table}[h]

\begin{tabular}{| l | c | c| }
\hline
\renewcommand*{\arraystretch}{1}
 & $B_i(K_m)$ & $T_i(K_m)$ \\ \hline
Scenario 1 &  $B_1(K_m) = 1$ & $T_1(K_m) = \left\{ \begin{array}{ccccc}
1 & \text{if} & m & \text{is} & \text{odd,} \\
2 & \text{if} & m & \text{is} & \text{even}
\end{array} \right.$\\ \hline
Scenario 2 & $B_2(K_m) =  \left\{ \begin{array}{ccccc}
1 & \text{if} & m & \text{is} & \text{even,} \\
2 & \text{if} & m & \text{is} & \text{odd}
\end{array} \right.$ & $T_2(K_m) = \left\{ \begin{array}{ccccc}
2 & \text{if} & m & \text{is} & \text{even,} \\
3 & \text{if} & m & \text{is} & \text{odd}
\end{array} \right.$ \\ \hline
Scenario 3 & $B_3(K_m) = m - 1$ & $T_3(K_m) = m$ \\ \hline 
\end{tabular}
\caption{Minimum tile and bond-edge type values for the complete graph family \cite{mintiles}} 
\label{table:Complete Graph}
\end{table}

\begin{remark} The graph $L_{3,n}$ is both a lollipop and a tadpole graph since $K_3 \cong C_3$. Results for this graph more naturally align with results for other tadpole graphs, and so these results are included in Section \ref{Tadpolegraphsection}. For the remainder of this work it is assumed $m >3$ for any $L_{m,n}$ graph. \end{remark}

\subsection{Scenario 1}

Recall that Scenario 1 represents the least restrictive scenario, in which all other graphs are allowable as constructions from a given pot realizing the target graph.

\begin{remark}
For all graphs $G$, $B_1(G) = 1$ \cite{mintiles}.
\end{remark}

\begin{proposition} \label{propscen1even1}
$$T_1(L_{m,n}) = \left\{ \begin{array}{ccccc}
3 & \text{for} & n = 1 \\
4 & \text{for} & n >1
\end{array} \right.$$
\end{proposition} 

\begin{proof}
\textit{Claim 1}: $T_1(L_{m,1}) = 3$. \par
In the $K_m$ subgraph, there are $m-1$ vertices of degree $m - 1$ and one vertex of degree $m$ (the bridging vertex), so the two degrees of vertices found in the $K_m$ subgraph will have opposite parity. When $n=1$ the path subgraph will consist of a single vertex with an odd degree of 1. Therefore, $av(L_{m,1}) = 3$, $ov(L_{m,1}) = 2$, and $ev(L_{m,1}) = 1$. By Theorem 1 of \cite{mintiles}, $3 \leq T_1(L_{m,1}) \leq 5$. The lower bound is realized by the following pots for even and odd $m$ values. Example labelings using directed edges to illustrate bond-edge orientation are shown in Figure \ref{fig:lollipopS1claim2}.
\begin{equation} \label{eq:Scen1even} P_{(2k,1)} = \{t_1 = \{\hat{a}^m\}, t_2 = \{ \hat{a}^{\frac{m}{2}- 1}, a^{\frac{m}{2}}\}, t_3 = \{\hat{a}\}\} \end{equation}

\begin{equation} \label{eq:Scen1odd} P_{(2k+1,1)} = \{t_1 = \{\hat{a}^{\lfloor{\frac{m}{2}} \rfloor}, a^{\lceil{\frac{m}{2}} \rceil}\}, t_2 = \{ \hat{a}^{\lfloor{\frac{m}{2}} \rfloor}, a^{\lfloor{\frac{m}{2}} \rfloor}\}, t_3 = \{\hat{a}\}\}\end{equation}

\begin{figure}[h!]
     \centering
	 \begin{tikzpicture}[transform shape, scale = 0.7]
 \node[main node] (a) at (0,0) {$t_1$};
 \node[main node] (b) at (-1,-2) {$t_2$};
 \node[main node](c) at (-3,-2) {$t_2$};
 \node[main node] (d) at (-4,0) {$t_2$};
 \node[main node] (e) at (-3,2) {$t_2$};
 \node[main node] (f) at (-1,2) {$t_2$};
 \node[main node] (g) at (2,0) {$t_3$};
		 
\path[draw,thick,color=red,->]
(a) edge node []{} (c)
(a) edge node []{} (d)
(a) edge node []{} (e)
(a) edge node []{} (f)
(a) edge node []{} (g)
(a) edge node []{} (b)
(b) edge node []{} (d)
(b) edge node []{} (f)
(c) edge node []{} (b)
(c) edge node []{} (d)
(d) edge node []{} (e)
(d) edge node []{} (f)
(e) edge node []{} (b)
(e) edge node []{} (c)
(f) edge node []{} (c)
(f) edge node []{} (e);

 \node[main node] (h) at ( 9,0) {$t_1$};
 \node[main node] (i) at (7.62
, -1.90) {$t_2$};
 \node[main node](j) at (5.38
,-1.17) {$t_2$};
 \node[main node] (k) at (5.38
,1.17) {$t_2$};
 \node[main node] (l) at (7.62, 1.90
) {$t_2$};
\node[main node] (m) at ( 11,0) {$t_3$};

\path[draw,thick,color=red,->]
(h) edge node []{} (m)
(h) edge node []{} (j)
(h) edge node []{} (l)
(i) edge node []{} (h)
(i) edge node []{} (k)
(j) edge node []{} (i)
(j) edge node []{} (l)
(k) edge node []{} (j)
(k) edge node []{} (h)
(l) edge node []{} (k)
(l) edge node []{} (i)
;
\end{tikzpicture}

\caption{Scenario 1 labeling of $L_{6,1}$ and $L_{5,1}$} 
\label{fig:lollipopS1claim2}
\end{figure}
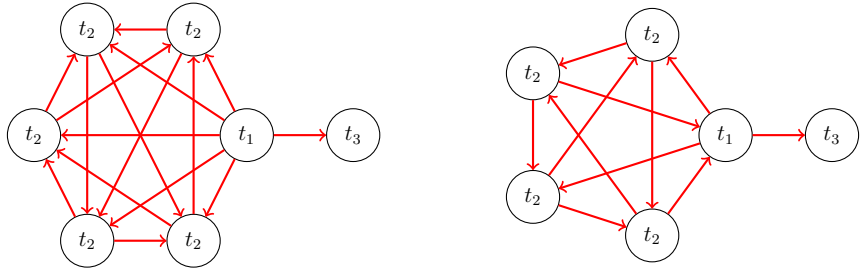

\textit{Claim 2}: $T_1(L_{m,n}) = 4$ for $n > 1$.\par
In $L_{m,n}$ with $n>1$, there are $m-1$ vertices of degree $m-1$, a single vertex of degree $m$, $n-1$ vertices of degree 2, and a single vertex of degree 1 at the end of the path. Thus, $ov(L_{m,n}) = 2$, and $ev(L_{m,n}) = 2$. By Theorem 1 of \cite{mintiles}, $4 \leq T_1(L_{m,n}) \leq 6$. The lower bound is realized by the following pots for even and odd $m$ values. Example labelings of $L_{4,2}$ and $L_{5,3}$ are shown in Figure \ref{fig:lollipopS1claim3}.
$$P_{(2k,n)} = \{t_1 = \{\hat{a}^m\}, t_2 = \{ \hat{a}^{\frac{m}{2}- 1}, a^{\frac{m}{2}}\}, t_3 = \{a, \hat{a}\}, t_4 = \{\hat{a}\}\}$$
\begin{equation*}  P_{(2k+1,n)} = \{t_1 = \{\hat{a}^{\lfloor{\frac{m}{2}} \rfloor}, a^{\lceil{\frac{m}{2}} \rceil}\}, t_2 = \{ \hat{a}^{\lfloor{\frac{m}{2}} \rfloor}, a^{\lfloor{\frac{m}{2}} \rfloor}\}, t_3 = \{\hat{a}, a\}, t_4 = \{\hat{a}\}\} \end{equation*}

\begin{figure}[h!]
     \centering
	 \begin{tikzpicture}[transform shape, scale = 0.7]
 \node[main node] (a) at (0,0) {$t_1$};
 \node[main node] (b) at (-2,-2) {$t_2$};
 \node[main node] (c) at (-4,0) {$t_2$};
 \node[main node] (d) at (-2,2) {$t_2$};
 \node[main node] (e) at (2,0) {$t_3$};
 \node[main node] (f) at (4,0) {$t_4$};
		 
\path[draw,thick,color=red,->]
(a) edge node []{} (b)
(a) edge node []{} (c)
(a) edge node []{} (d)
(a) edge node []{} (e)
(b) edge node []{} (c)
(c) edge node []{} (d)
(d) edge node []{} (b)
(e) edge node []{} (f);
\end{tikzpicture} \hspace{8mm}
\begin{tikzpicture}[transform shape, scale = 0.7]
\node[main node] (h) at ( 9,0) {$t_1$};
 \node[main node] (i) at (7.62
, -1.90) {$t_2$};
 \node[main node](j) at (5.38
,-1.17) {$t_2$};
 \node[main node] (k) at (5.38
,1.17) {$t_2$};
 \node[main node] (l) at (7.62, 1.90
) {$t_2$};
\node[main node] (m) at ( 11,0) {$t_3$};
\node[main node] (n) at ( 13,0) {$t_3$};
\node[main node] (o) at ( 15,0) {$t_4$};

\path[draw,thick,color=red,->]
(h) edge node []{} (m)
(h) edge node []{} (j)
(h) edge node []{} (l)
(i) edge node []{} (h)
(i) edge node []{} (k)
(j) edge node []{} (i)
(j) edge node []{} (l)
(k) edge node []{} (j)
(k) edge node []{} (h)
(l) edge node []{} (k)
(l) edge node []{} (i)
(m) edge node []{} (n)
(n) edge node []{} (o)
;
\end{tikzpicture}
\caption{Scenario 1 labeling of $L_{4,2}$ and $L_{5,3}$} 
\label{fig:lollipopS1claim3}
\end{figure}
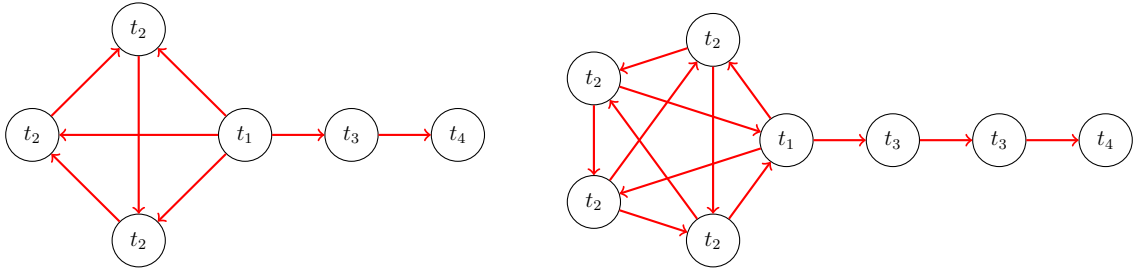
\end{proof}

\subsection{Scenario 2}\label{sec:LollipopS2}
Recall that Scenario 2 requires that no graphs of smaller order than the target graph can be realized by the proposed pot. \par

\begin{proposition}\label{prop:LollipopEvenS2}
$B_{2}(L_{2k, n}) = n$ and $T_{2}(L_{2k, n}) = n+2$ for $n \leq 2k$.
\end{proposition}

\begin{proof} 
By Lemma \ref{S2NoRepeatedBondEdge} $B_2(L_{2k,n})\geq n$ and $T_2(L_{2k,n}) \geq n$, with at least $n$ distinct tile types needed in the path. The vertices in the path are of different degrees than the vertices in the $K_{2k}$ subgraph, so at least two additional tile types are needed. Therefore, $T_2(L_{2k,n}) \geq n+2$. The lower bounds are achieved by the following pot. An example labeling of $L_{6,3}$ is shown in Figure \ref{fig:L6,3Scen2}.
\begin{equation}\label{eq:LollipopEvenS2pot} \begin{split} P = \{t_1 = \{a^{2k-1}_1, a_2\}, t_2 = \{ a_1^{k- 1}, \hat{a}^{k}_1\}, t_i = \{\hat{a}_{i-1}, a_i\} \text{ for } 3 \leq i \leq  n, \\ t_{n+1}=\{\hat{a}_n,a_1\},   t_{n+2} = \{\hat{a}_1\}\} \end{split} \end{equation}

The construction matrix and spectrum of $P$ follow.
$$M(P) = \begin{bmatrix} 
2k-1 & -1 & 0 & \cdots & 0 & 1 & -1 & 0  \\ 
1 & 0 & -1 & \ddots & 0 & 0 & 0 & 0 \\ 
0 & 0 & 1 & \ddots & 0 & 0 & 0 & 0  \\ 
\vdots & \vdots &  & \ddots & -1 & 0 & 0 & 0 \\
0 & 0 & \cdots & 0 & 1  & -1 & 0 & 0 \\
1 & 1 & 1 & \cdots & 1 & 1 & 1 & 1 \end{bmatrix}$$

\begin{equation*} \begin{split} \mathcal{S}(P)=\left\{\left\langle r_1 = \frac{1}{2k + n}, r_2 = \frac{2k}{2k+n} - r_{n+2},  r_3 = \frac{1}{2k+n} \text{ for } 3 \leq i \leq n+1, \right\rangle \mid r_{n+2} \in \mathbb{Q}^+\right\}. \end{split} \end{equation*} 

The least common multiple of the vector component denominators in $\mathcal{S}(P)$ is at least $2k+ n$.  Therefore, no graphs of order smaller than $L_{2k,n}$ are realized by the pot.
\end{proof}

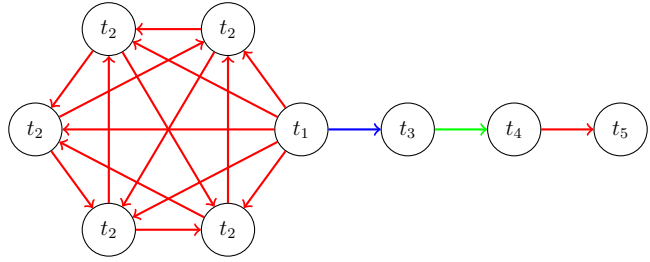
\begin{figure}[h!]
     \centering
 \begin{tikzpicture}[transform shape, scale = 0.7]

 \node[main node] (a) at ( 9,0) {$t_1$};
 \node[main node] (b) at (7.62
, -1.90) {$t_2$};
 \node[main node](c) at (5.38
,-1.90) {$t_2$};
 \node[main node] (d) at (4,0) {$t_2$};
 \node[main node] (e) at (5.38
,1.90) {$t_2$};
 \node[main node] (f) at (7.62, 1.90
) {$t_2$};
\node[main node] (g) at ( 11,0) {$t_3$};
\node[main node] (h) at ( 13,0) {$t_4$};
\node[main node] (i) at ( 15,0) {$t_5$};

\path[draw,thick,color=red,->]
(a) edge node []{} (b)
(a) edge node []{} (c)
(a) edge node []{} (d)
(a) edge node []{} (e)
(a) edge node []{} (f)
(b) edge node []{} (d)
(e) edge node []{} (b)
(c) edge node []{} (b)
(c) edge node []{} (e)
(d) edge node []{} (c)
(e) edge node []{} (d)
(b) edge node []{} (f)
(f) edge node []{} (c)
(f) edge node []{} (e)
(d) edge node []{} (f)
(h) edge node []{} (i);

\path[draw,thick,color=blue,->]
(a) edge node []{} (g);

\path[draw,thick,color=green,->]
(g) edge node []{} (h);
\end{tikzpicture}

\caption{Scenario 2 labeling of $L_{6,3}$ } 
\label{fig:L6,3Scen2}
\end{figure}

\begin{example}\label{ex:L2k,1} The graphs $L_{2k,1}$ are a special case of Proposition \ref{prop:LollipopEvenS2}, since there are no degree 2 vertices in the path, only the vertex of degree 1; thus only one bond-edge type will be used as implied by the pot given in (\ref{eq:LollipopEvenS2pot}). Here we illustrate concretely that $B_{2}(L_{2k, 1}) = 1$ and $T_{2}(L_{2k, 1}) = 3$.\par 

We claim that the pot achieving the value of $T_{1}(L_{2k,n})$ given in (\ref{eq:Scen1even}), which aligns with the pot given in (\ref{eq:LollipopEvenS2pot}), also satisfies the requirements of Scenario 2 and thus provides the value for  $T_{2}(L_{2k, 1})$ as well. The construction matrix and spectrum of $P$ show that nothing smaller than the target graph can be realized by the pot. 
$$M(P) = \begin{bmatrix} 
2k & -1 & -1 & 0  \\ 
1 & 1  & 1 & 1 \end{bmatrix}$$

$$\mathcal{S}(P) = \left\{ \left\langle \frac{1}{2k + 1}, \frac{2k}{2k + 1} - r_3, r_3
\right\rangle \mid r_3 \in \mathbb{Q}^+ \right\}$$
The least common multiple of the vector component denominators in $\mathcal{S}(P)$ is at least $2k+1$, which is the order of the target graph. Therefore, by Proposition 3 of \cite{mintiles}, nothing smaller can be realized by the pot.
\end{example}

Unlike in the case of $L_{2k,1}$ as demonstrated in Example \ref{ex:L2k,1}, lollipop graphs with an odd complete graph subgraph and length 1 path do not follow the same pattern for minimum number of bond-edge types as those with paths of length at least 2. Hence, we provide the next proposition as a separate case.

\begin{proposition} \label{prop:LollipopS2path1} 
$B_2(L_{2k+1,1})=2$.
\end{proposition}
\begin{proof}
In a similar fashion to the proof that $B_2(K_{2k+1})\geq 2$ given in \cite{mintiles}, we argue that $B_2(L_{2k+1,1})\geq 2$. By way of contradiction, assume there is a pot with just one bond-edge type realizing $L_{2k+1,1}$. The construction matrix of $P$ must have the form
$$M(P) = \begin{bmatrix} z_{1,1} & z_{1,2} & \cdots & z_{1,p} & 0 \\ 1 & 1 & \cdots & 1 & 1 \end{bmatrix}$$
Note that $z_{1,j} \neq 0$ for all $j$, since two tile types will be of odd degree, and the remaining tile types of even degree cannot have an equal number of hatted and un-hatted cohesive-ends. Such a tile could realize a graph of order 1 with $k$ loop edges. Re-ordering tile numbers if necessary, we may assume $t_{p-1}$ and $t_p$ are the tiles of odd degrees $1$ and $2k+1$.\par
Furthermore, we claim $z_{1,1},...,z_{1,p-2}$ cannot all have the same sign. For the sake of contradiction and without loss of generality, assume that $z_{1,j}>0$ for all $1 \leq j \leq p-2$. Then each of $t_{1},...,t_{p-2}$ must have at least $k+1$ arms labeled with the un-hatted version of the bond-edge type. In order for this pot to realize $L_{2k+1,1}$, this results in at least $2k(k+1)=2k^2+2k$ half-edges in the graph labeled with the un-hatted version of the bond-edge type. There are $4k^2+2k+2$ total half-edges in the graph, which leaves $2k^2+2<2k^2+2k$ edges to potentially be labeled with the hatted version of the bond-edge type, an insufficient amount. Therefore we may also re-order tile numbers such that $z_{1,1}>0$ and $z_{1,2}<0$. Thus, $M(P)$ is row equivalent to
$$\begin{bmatrix} 1 & 0 & \frac{-z_{1,3}+z_{1,2}}{z_{1,1}-z_{1,2}} & \cdots & \frac{-z_{1,p}+z_{1,2}}{z_{1,1}-z_{1,2}} & \frac{-z_{1,2}}{z_{1,1}-z_{1,2}} \\ 0 & 1 & \frac{z_{1,1}-z_{1,3}}{z_{1,1}-z_{1,2}} & \cdots & \frac{z_{1,1}-z_{1,p}}{z_{1,1}-z_{1,2}} & \frac{z_{1,1}}{z_{1,1}-z_{1,2}} \end{bmatrix}$$
and so has a solution of the form $\langle -z_{1,2}/(z_{1,1}-z_{1,2}), z_{1,1}/(z_{1,1}-z_{1,2}),0,...,0\rangle$. Since $z_{1,1}$ and $z_{1,2}$ are both non-zero, even, and have absolute value less than or equal to $2k$, this solution has the form $\langle a/n,b/n,0,...0 \rangle$ where $n < 2k+1$. Thus, $P$ realizes a graph of order $n<2k+1$.\par
The following pot with two bond-edge types realizes $L_{2k+1,1}$. An example labeling of $L_{5,1}$ is shown in Figure \ref{fig:L5,1Scen2}.
\begin{equation*}  P = \{t_1 = \{a_1^{2k+1}\}, t_2 = \{\hat{a_1}, \hat{a}^{\lfloor{\frac{2k+1}{2}\rfloor}}_2, a^{\lfloor\frac{2k+1}{2}\rfloor - 1}_2\}, t_3 = \{\hat{a_1}, a^{\lfloor{\frac{2k+1}{2}\rfloor}}_2, \hat{a}^{\lfloor\frac{2k+1}{2}\rfloor - 1}_2\}, t_4 = \{\hat{a_1}\}\}\end{equation*} 

The spectrum of $P$ follows. 

 $$\mathcal{S}(P)= \left\{\left\langle \frac{1}{2k + 2}, \frac{2k+1}{2(2k + 1)} - \frac{1}{2}r_4, \frac{2k+1}{2(2k + 1)} - \frac{1}{2}r_4, r_4 \right\rangle \mid r_4 \in \mathbb{Q}^+\right\}$$ 
 
The least common multiple of the vector component denominators in $\mathcal{S}(P)$ is at least $2k+2$. Therefore, no graphs of smaller order are realized by the pot. 
\end{proof}

\begin{figure}[h!]
     \centering
 \begin{tikzpicture}[transform shape, scale = 0.7]

 \node[main node] (a) at ( 9,0) {$t_1$};
 \node[main node] (b) at (7.62
, -1.90) {$t_2$};
 \node[main node](c) at (5.38
,-1.17) {$t_2$};
 \node[main node] (d) at (5.38
,1.17) {$t_3$};
 \node[main node] (e) at (7.62, 1.90
) {$t_3$};
\node[main node] (f) at ( 11,0) {$t_4$};

\path[draw,thick,color=red,->]
(a) edge node []{} (b)
(a) edge node []{} (c)
(a) edge node []{} (d)
(a) edge node []{} (e)
(a) edge node []{} (f);

\path[draw,thick,color=blue,->]
(b) edge node []{} (d)
(b) edge node []{} (e)
(c) edge node []{} (b)
(c) edge node []{} (e)
(d) edge node []{} (c)
(e) edge node []{} (d);
\end{tikzpicture}

\caption{Scenario 2 labeling of $L_{5,1}$}
\label{fig:L5,1Scen2}
\end{figure}
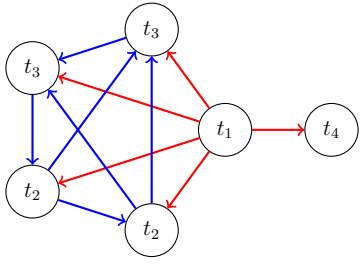

\begin{proposition}\label{prop:LollipopOddB2} $B_2(L_{2k+1,n})= n$ for $2 \leq n \leq 2k+1$. \end{proposition}

\begin{proof} 
By Lemma \ref{S2NoRepeatedBondEdge}, $B_2(L_{2k+1,n})\geq n$. The lower bound is achieved by the following pot. An example labeling of $L_{5,5}$ is shown in Figure \ref{fig:L5,5Scen2}. 

\begin{equation}\label{eq:LollipopOddB2medpathpot}  \begin{split} P = \{t_1 = \{a_1^{2k+1}\}, t_2 = \{\hat{a_1}, a^{k - 1}_2, \hat{a_2}^{k}\}, t_3 = \{\hat{a_1}, a^{k}_2, \hat{a}^{k - 1}_2\},\\ t_i = \{\hat{a}_{i-3}, a_{i-2}\} \text{ for } i = 4,...,n+2, t_{n+3} = \{\hat{a_{n}}\}\} \end{split} \end{equation} 

The construction matrix and spectrum of $P$ follow.

$$M(P) = \begin{bmatrix} 
2k+1 & -1 & -1 & -1 & 0 & \cdots & & \cdots & 0 \\ 
0 & -1 & 1 & 1 & -1 & 0 & \cdots & & 0 \\ 
\vdots & 0 & 0 & 0 & 1 & -1 & \ddots & & \vdots \\
 & &  &  & \ddots & \ddots & \ddots  &  &  \\
0 & \cdots & &  & \cdots & 0 & 1 & -1 & 0 \\
1 & 1 & \cdots & & \cdots & 1 & 1 & 1 & 1 \end{bmatrix}$$

 \begin{equation*} \begin{split} \mathcal{S}(P)=\left\{\left\langle r_1 = \frac{1}{2k+2} - \frac{n-1}{2k+2}(r_{n+3}), r_2 = \frac{2k+1}{2(2k+1)+2} - \frac{(2k+1)n+1}{2(2k+1)+2}(r_{n+3}), \right. \right. \\ r_3 = \frac{2k+1}{2(2k+1)+2} - \frac{(2k+1)n-2(2k+1)-1}{2(2k+1)+2}(r_{n+3}) - r_4, r_4, \\ \left. \left. r_i = r_{n+3} \text{ for } 5 \leq i \leq n+2, r_{n+3} \right\rangle \mid r_4, r_{n+3} \in \mathbb{Q}^+\right\} \end{split} \end{equation*} 

Note that the spectrum of this pot has two degrees of freedom, making it more difficult to use the spectrum to show a solution constructing a graph smaller than $L_{2k+1,n}$ does not exist. Instead, we proceed by considering the restrictive nature of the tile types.\par
First note that no complete complex can be realized by $P$ without the use of tile type $t_1$, since $t_1$ is the only tile with an arm labeled $a_1$ and any combination of other tiles bonded together will eventually include arms labeled $\hat{a_1}$. This can be further verified by setting $r_1=0$ in $\mathcal{S}(P)$, which gives a negative, and therefore invalid, tile proportion for $r_2$. Since $t_1$ must be included, any complete complex is of size at least $2k+2$, as $t_1$ has $2k+1$ arms that cannot form loop edges. \par

Suppose some subset of $\{t_2, t_3, t_4\}$ bonds to the $2k+1$ arms of $t_1$. If $t_5$ then bonds to one of $\{t_2, t_3, t_4\}$ (see Figure \ref{fig:B2oddLollprop}), the complex is of size at least $2k+1+1+(n-1) =2k+n+1$, the target graph order.

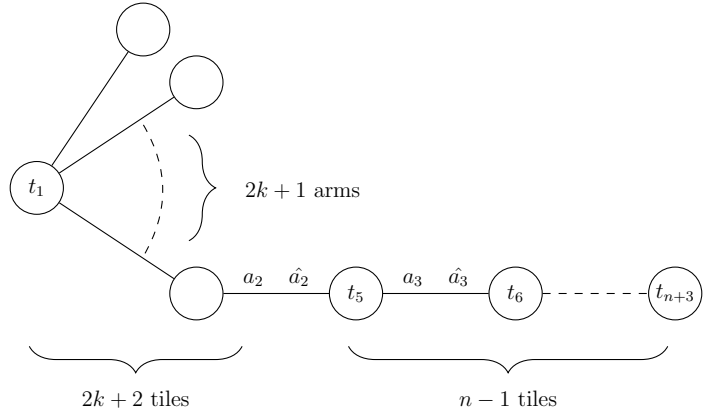
\begin{figure}[h!]
\centering
	 \begin{tikzpicture}[transform shape, scale = 0.7]
 \node[main node] (a) at (0,0) {$t_1$};
 \node[main node] (b) at (2,3) {};
 \node[main node] (c) at (3,2) {};
 \node[main node] (d) at (3,-2) {};
 \node[main node, fill=white, draw=white, minimum size = 0cm] (e) at (2,1.2) {};
 \node[main node, fill=white, draw=white, minimum size = 0cm] (f) at (2,-1.3) {};
\node[main node] (g) at (6,-2) {$t_5$};
\node[main node] (h) at (9,-2) {$t_6$};
\node[main node] (i) at (12,-2) {$t_{n+3}$};

 \node[main node, fill=white, draw=white, minimum size = 0cm] (j) at (5,-0.05) {$2k+1$ arms};

\path[draw]
(a) edge node []{} (b)
(a) edge node []{} (c)
(a) edge node []{} (d)
(d) edge node [above]{$a_2$ \hspace{2mm} $\hat{a_2}$} (g)
(g) edge node [above]{$a_3$ \hspace{2mm} $\hat{a_3}$} (h);

\path[draw, dashed]
(e) edge [bend left] node []{} (f)
(h) edge node []{} (i);

\draw [decorate,decoration={brace,amplitude=10pt},xshift=-4pt,yshift=0pt]
(4,-3) -- (0,-3) node [black,midway,xshift=0, yshift=-1cm] 
{$2k+2$ tiles};

\draw [decorate,decoration={brace,amplitude=10pt},xshift=-4pt,yshift=0pt]
(12,-3) -- (6,-3) node [black,midway,xshift=0, yshift=-1cm] 
{$n-1$ tiles};

\draw [decorate,decoration={brace,amplitude=10pt},xshift=-4pt,yshift=0pt]
(3,1) -- (3,-1) node [black,midway,xshift=10, yshift=0cm] {};
\end{tikzpicture}

\caption{Complex formed in proof of Proposition \ref{prop:LollipopOddB2}; $2k+1$ tiles bonded to $t_1$ must be $t_2, t_3 \text{ or } t_4$} 
\label{fig:B2oddLollprop}
\end{figure}

Note that, due to the bond-edge types on the tile arms, $t_i$ for $5 \leq i \leq n+3$ can only bond to the complex if $t_5$ has already been included. Thus, if $t_5$ is not included, the only possibility remaining for construction of a graph of order less than the target graph is if the complex is constructed from the pot $P' = \{t_1, t_2, t_3, t_4\}$. It is easy to see from $\mathcal{S}(P)$ that  $\mathcal{S}(P')$ is as follows.
 \begin{equation*} \begin{split} \mathcal{S}(P')=\left\{\left\langle r_1 = \frac{1}{2k+2}, r_2 = \frac{2k+1}{2(2k+1)+2}, r_3 = \frac{2k+1}{2(2k+1)+2} - r_4, r_4 \right\rangle \mid r_4 \in \mathbb{Q}^+\right\} \end{split} \end{equation*} 
Note that $2k+1$ shares no common factors with $2$ or $2k+2$, hence The least common multiple of the vector component denominators in $\mathcal{S}(P')$ is at least $2(2k+1)+2$, which is greater than $2k+n+1$, the target graph order, since $n \leq 2k+1$.
\end{proof}

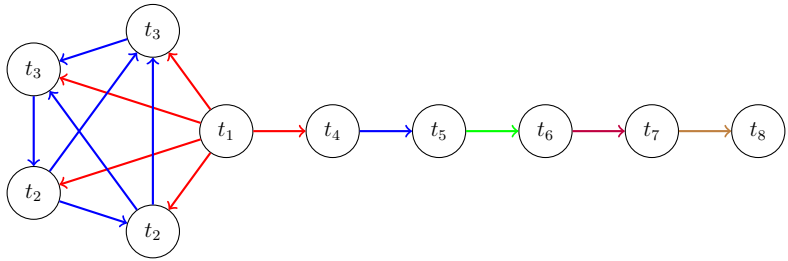
\begin{figure}[h!]
     \centering
 \begin{tikzpicture}[transform shape, scale = 0.7]

 \node[main node] (a) at ( 9,0) {$t_1$};
 \node[main node] (b) at (7.62
, -1.90) {$t_2$};
 \node[main node](c) at (5.38
,-1.17) {$t_2$};
 \node[main node] (d) at (5.38
,1.17) {$t_3$};
 \node[main node] (e) at (7.62, 1.90
) {$t_3$};
\node[main node] (f) at ( 11,0) {$t_4$};
\node[main node] (g) at ( 13,0) {$t_5$};
\node[main node] (h) at ( 15,0) {$t_6$};
\node[main node] (i) at ( 17,0) {$t_7$};
\node[main node] (j) at ( 19,0) {$t_8$};

\path[draw,thick,color=red,->]
(a) edge node []{} (b)
(a) edge node []{} (c)
(a) edge node []{} (d)
(a) edge node []{} (e)
(a) edge node []{} (f);

\path[draw,thick,color=blue,->]
(b) edge node []{} (d)
(b) edge node []{} (e)
(c) edge node []{} (b)
(c) edge node []{} (e)
(d) edge node []{} (c)
(e) edge node []{} (d)
(f) edge node []{} (g);

\path[draw,thick,color=green,->]
(g) edge node []{} (h);

\path[draw,thick,color=purple,->]
(h) edge node []{} (i);

\path[draw,thick,color=brown,->]
(i) edge node []{} (j);
\end{tikzpicture}

\caption{Scenario 2 labeling of $L_{5,5}$} 
\label{fig:L5,5Scen2}
\end{figure}

Unlike for $L_{2k,n}$ graphs, we have chosen to address $T_2(L_{2k+1,n})$ for $n \leq 2k+1$ in a separate proposition rather than combining the result with Proposition \ref{prop:LollipopOddB2}, since justification for the necessity of two different tile types in labeling the $K_{2k+1}$ subgraph is non-trivial.

\begin{proposition}\label{prop:LollipopOddT2} $T_2(L_{2k+1,n})=n+3$ for $n \leq 2k+1$.
\end{proposition}

\begin{proof}
By Lemma \ref{S2NoRepeatedBondEdge}, at least $n$ distinct tile types must be used in the path. Assume exactly one tile type of degree $2k$ is used to label the $K_{2k+1}$ subgraph. Note that for any bond-edge type used in the labeling, the following equation must be satisfied:
$$z_1+2kz_2=0,$$
where $z_1, z_2$ are the net number of cohesive-ends on the arms of the bond-edge type on the degree $2k+1$ tile and degree $2k$ tile, respectively. Note that $z_2 \neq 0$ for at least one bond-edge type used in the labeling, since otherwise the tile used for the vertex of degree $2k$ could form a graph of order 1 with $k$ loops. Let $a$ be a bond-edge type used in the labeling such that $z_2 \neq 0$. Note $|z_1| \leq 2k+1$, so $|z_2| \leq 1$ and thus $|z_2|=1$. Without loss of generality, let $z_2=1$. Since degree $2k$ vertices are adjacent to one another in the graph, and there are $2k-1$ arms remaining to be labeled on the degree $2k$ tile, it is impossible to label those arms in such a way that the net number of cohesive-ends of any other other bond-edge types will be 0 and the net number of cohesive-ends of bond-edge type $a$ remains at a value of 1. Therefore, at least two distinct tile types of degree $2k$ are needed in the construction of the $K_{2k+1}$ subgraph. Another tile type of degree $2k+1$ is needed for the vertex of the $K_{2k+1}$ subgraph that is adjacent to a vertex in the path. Thus, $T_2(L_{2k+1,n}) \geq n+3$. \par
The pot given in (\ref{eq:LollipopOddB2medpathpot}) achieves the lower bound for all $L_{2k+1,n}$ graphs. 
\end{proof}

\begin{proposition}\label{prop:LollipopEvenS2mediumpath} $\lceil \frac{n-2k+1}{2} \rceil + 2k -1 \leq B_2(L_{2k,n}) \leq n$ and $\lceil \frac{n-2k+1}{2} \rceil +2k+2 \leq T_2(L_{2k,n}) \leq n+2$ for $n > 2k$.
\end{proposition}
\begin{proof} By Lemma \ref{S2RepeatedBondEdge} $B_2(L_{2k,n}) \geq \lceil \frac{n-2k+1}{2} \rceil + 2k -1$ and $T_2(L_{2k,n}) \geq \lceil \frac{n-2k+1}{2} \rceil +2k$, with $\lceil \frac{n-2k+1}{2} \rceil +2k$ distinct tile types needed to label the vertices of the path subgraph. The vertices of degrees $2k$ and $2k-1$ in the $K_{2k}$ subgraph require additional distinct tile types, so $T_2(L_{2k,n}) \geq \lceil \frac{n-2k+1}{2} \rceil +2k+2$. The pot given in (\ref{eq:LollipopEvenS2pot}) also realizes $L_{2k,n}$ when $n > 2k$, and the proof of Proposition \ref{prop:LollipopEvenS2} shows that this pot does not realize any smaller graphs. Therefore, $B_2(L_{2k,n}) \leq n$ and $T_2(L_{2k,n}) \leq n+2$.
\end{proof}

\begin{remark} Suppose $P$ is a pot realizing $L_{2k,n}$ in which the path is labeled to achieve a minimum number of bond-edge and tile types as described in Lemma \ref{S2RepeatedBondEdge}. The $K_{2k}$ subgraph remains labeled as in Figure \ref{fig:L6,3Scen2}, and no bond-edge types are used to label both an edge in the $K_{2k}$ subgraph and the path subgraph. Then $P$ would consist of $\lceil \frac{n-2k+1}{2} \rceil + 2k$ bond-edge types and $\lceil \frac{n-2k+1}{2} \rceil +2k+2$ tile types. In this case, $\mathcal{S}(P)$ would have one degree of freedom. If, instead, the bond-edge type used to label the $K_{2k}$ subgraph is used to label the edge incident with the degree 1 vertex of the path, as shown in Figure \ref{fig:L6,3Scen2}, then $\mathcal{S}(P)$ would have two degrees of freedom. For either pot, analysis in Scenario 2 would prove difficult. Narrowing the bounds given in Proposition \ref{prop:LollipopEvenS2mediumpath} remains an open problem.
\end{remark}

\begin{proposition}\label{prop:LollipopOddS2mediumpath} $\lceil \frac{n-2k}{2} \rceil +2k \leq B_2(L_{2k+1,n}) \leq n$ and $\lceil \frac{n-2k}{2} \rceil+2k+4 \leq T_2(L_{2k,n}) \leq n+3$ for $n > 2k+1$.
\end{proposition}
\begin{proof} By Lemma \ref{S2RepeatedBondEdge} $B_2(L_{2k+1,n}) \geq \lceil \frac{n-2k}{2} \rceil + 2k$ and $T_2(L_{2k+1,n}) \geq \lceil \frac{n-2k}{2} \rceil +2k+1$, with $\lceil \frac{n-2k}{2} \rceil +2k+1$ distinct tile types needed to label the vertices of the path subgraph. The proof of Proposition \ref{prop:LollipopOddT2} shows that at least three distinct tile type are necessary to label the vertices of the $K_{2k+1}$ subgraph, so $T_2(L_{2k+1,n}) \geq \lceil \frac{n-2k}{2} \rceil +2k+4$. The pot given in (\ref{eq:LollipopOddB2medpathpot}) also realizes $L_{2k+1,n}$ when $n > 2k+1$, and the proof of Proposition \ref{prop:LollipopOddB2} shows that this pot does not realize any smaller graphs. Therefore, $B_2(L_{2k+1,n}) \leq n$ and $T_2(L_{2k+1,n}) \leq n+3$.
\end{proof}

\begin{remark} Suppose $P$ is a pot realizing $L_{2k+1,n}$ in which the path is labeled to achieve a minimum number of bond-edge and tile types as described in Lemma \ref{S2RepeatedBondEdge}, the $K_{2k+1}$ subgraph remains labeled as in Figure \ref{fig:L5,5Scen2}, and no bond-edge types are used to label both an edge in the $K_{2k+1}$ subgraph and the path subgraph. Then, $P$ would consist of $\lceil \frac{n-2k}{2} \rceil + 2k+2$ bond-edge types and $\lceil \frac{n-2k}{2} \rceil +2k+4$ tile types. In this case, $\mathcal{S}(P)$ would have one degree of freedom. It remains an open question as to whether the bond-edge types used to label the $K_{2k+1}$ subgraph can be used to label edges of the path subgraph while bond-edge types are also repeated along the path, which could allow for a pot achieving the lower bound for $B_2(L_{2k+1,n})$ given in Proposition \ref{prop:LollipopOddS2mediumpath}. 
\end{remark}

\begin{proposition}
$m+n-2 \leq B_{3}(L_{m, n}) \leq m + n - 1$.
\end{proposition}
\begin{proof}
 The proof that $B_3(K_m)\geq m-1$ given in \cite{mintiles} applies to show that at least $m-1$ bond-edge types are needed in the labeling of the $K_m$ subgraph of $L_{m,n}$. As in $K_m$, all vertices in the $K_m$ subgraph of $L_{m,n}$ are adjacent to one another, so using fewer than $m-1$ bond-edge types to label the edges in the $K_m$ subgraph will allow for formation of a multiple-edge between two of those vertices. $L_{m,n}$ has no multiple-edges, so the resulting graph would be non-isomorphic to $L_{m,n}$. Note that this can occur in the labeling presented in Figure \ref{fig:L5,5Scen2} between vertices labeled with tile types $t_2$ and $t_3$. By Lemma \ref{S3NoRepeatedBondEdge} at least $n$ distinct bond-edge types must be used in the path. \par
Suppose a bond-edge type is used in the labeling of the $K_m$ subgraph is also used to label an edge in the path not incident with the bridging vertex in the $K_m$ subgraph. It can be easily verified that, regardless of labeling orientation, these two bond-edges can break and re-bond in such a way that there is no longer a path of length $n$ adjoined to the graph via a single cut-vertex, creating a non-isomorphic graph. Furthermore, if a bond-edge type is used to label both the edge in the path subgraph incident with the bridging vertex in the $K_m$ subgraph and an edge in the $K_m$ subgraph not incident with that edge, then a graph with a multiple-edge can form. The only remaining possibility for repetition of a bond-edge type between the $K_m$ subgraph and the path subgraph is if the bond-edge type used to label the edge of the path incident with the $K_m$ vertex is used to label another edge incident with that vertex. Since at most a single bond-edge type can be used in the labeling of both the $K_m$ subgraph and the path subgraph, $B_3(L_{m,n}) \geq m+n-2$.  \par 
The upper bound of $m+n-1$ is achieved by following pot. An example labeling of $L_{5,3}$ using this pot is given in Figure \ref{fig:L5,3Scen3}.

\begin{equation}\label{eq:LollipopS3B3pot} \begin{split}  P = \{t_1 = \{a_1^{m-1}, a_m\}, t_i = \{\hat{a}_{1}, ..., \hat{a}_{i-1}, a_i^{m-i}\} \text{ for } 2 \leq i \leq m-1, t_m = \{\hat{a}_1,..., \hat{a}_{m-1}\}\\ t_{m+1} = \{\hat{a}_m, a_{m+1}\}, t_l = \{\hat{a}_{l-1}, a_{l}\}, \text{ for } m + 2 \leq l \leq m + n - 1, t_{m+n} = \{\hat{a}_{m+n-1}\}\} \end{split} \end{equation}

The construction matrix and spectrum of $P$ follow. 

\begin{equation*} M(P) = \begin{bmatrix} 
m-1 & -1 & \cdots & &  & \cdots & -1 & 0 & \cdots & & 0\\ 
0 & m-2 & -1 & \cdots & & -1 & 0 & 0 & \cdots & & 0 \\ 
\vdots & 0 & m-3 & -1 & & \vdots & \vdots & \vdots & & & \vdots\\
 & \vdots & \ddots & \ddots & \ddots & & & & & & \\
 & & & 0 & 1 & -1  & 0 & 0 & \cdots & & 0\\
 & & & & 0 & 0 & 1 & -1 & 0 & \cdots & 0\\
\vdots & \vdots & & & & \ddots & \ddots & \ddots & \ddots & \ddots & \vdots \\
0 & 0 & \cdots & & & & 0 & 0 & 1 & -1 & 0 \\
1 & 1 & \cdots & & & & & & \cdots & 1 & 1 \end{bmatrix} \end{equation*}

 \begin{equation*} \mathcal{S}(P)=\left\{\left\langle r_i = \frac{1}{m+n} \text{ for } 1 \leq i \leq m+n \right\rangle\right\} \end{equation*} 
 
The least common multiple of the single vector component denominator in $\mathcal{S}(P)$ is $m+n$, so no graphs of order smaller than $L_{m,n}$ are realized by the pot. To see that no graph of order $m+n$ not isomorphic to $L_{m,n}$ may be realized by this pot, note that the unique solution to the construction matrix implies that any such graph must use exactly the same numbers of each tile type as $L_{m,n}$ and the only possible re-combination of bond-edge formations represent simple ``swaps" of edges incident with the same vertex. 
\end{proof}

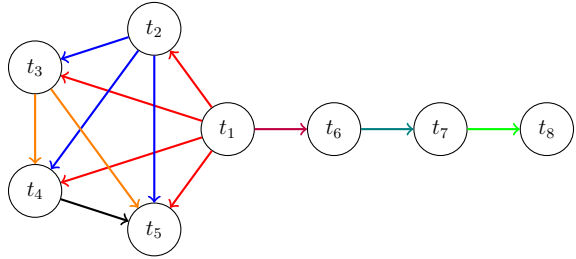
\begin{figure}[h!]
     \centering
 \begin{tikzpicture}[transform shape, scale = 0.7]

 \node[main node] (a) at ( 9,0) {$t_1$};
 \node[main node] (b) at (7.62
, -1.90) {$t_5$};
 \node[main node](c) at (5.38
,-1.17) {$t_4$};
 \node[main node] (d) at (5.38
,1.17) {$t_3$};
 \node[main node] (e) at (7.62, 1.90
) {$t_2$};
\node[main node] (f) at ( 11,0) {$t_6$};
\node[main node] (g) at ( 13,0) {$t_7$};
\node[main node] (h) at ( 15,0) {$t_8$};

\path[draw,thick,color=red,->]
(a) edge node []{} (b)
(a) edge node []{} (c)
(a) edge node []{} (d)
(a) edge node []{} (e);

\path[draw,thick,color=blue,->]
(e) edge node []{} (b)
(e) edge node []{} (c)
(e) edge node []{} (d);

\path[draw,thick,color=orange,->]
(d) edge node []{} (b)
(d) edge node []{} (c);

\path[draw,thick,color=black,->]
(c) edge node []{} (b);

\path[draw,thick,color=purple,->]
(a) edge node []{} (f);

\path[draw,thick,color=teal,->]
(f) edge node []{} (g);

\path[draw,thick,color=green,->]
(g) edge node []{} (h);
\end{tikzpicture}

\caption{Scenario 3 labeling of $L_{5,3}$ } 
\label{fig:L5,3Scen3}
\end{figure}

Note that the pot given in (\ref{eq:LollipopS3B3pot}) will satisfy the requirements of Scenario 3 for any order $L_{m,n}$ graph, but for some $L_{m,n}$ graphs a pot with only $m+n-2$ bond-edge types can also satisfy the requirements of Scenario 3. In this case, the edges labeled with the same bond-edge type must have matching labeling orientations with respect to the connecting vertex, otherwise a loop edge can form. No immediate problem arises if the same bond-edge type with matching labeling orientation is used to label both the edge of the path incident with the $K_m$ subgraph vertex and another edge incident with that vertex. However, the possibility for realization of a smaller or non-isomorphic graph exists. Since each vertex of the $K_m$ subgraph is labeled exclusively with cohesive-ends for which matching cohesive-ends exists on other vertices of the $K_m$ subgraph, a complete complex can always form from the $m$ tile types used to label the vertices of the $K_m$ subgraph alone. The size of such a complex always has the potential to be smaller than the target graph if the value of $n$ is large enough, but the exact magnitude of $n$ required to prevent this violation of Scenario 2, and therefore of Scenario 3, varies depending on how the $K_m$ subgraph is labeled. This makes finding a pot for each order of $L_{m,n}$ challenging. Here we provide a few examples to illustrate this difficulty.

\begin{example} \label{ex:1LollipopOddS3}
Consider the following pot realizing $L_{m,n}$ in which the same bond-edge type is used to label all edges incident with the bridging vertex. An example labeling of $L_{5,6}$ using this pot is shown in Figure \ref{fig:L5,6Scen3ex1}. 

\begin{equation*} \begin{split}  P = \{t_1 = \{a_1^{m}\}, t_i = \{\hat{a}_{1}, ..., \hat{a}_{i-1}, a_i^{m-i}\} \text{ for } 2 \leq i \leq  m-1, t_m = \{\hat{a}_1,..., \hat{a}_{m-1}\}\\ t_{m+1} = \{\hat{a}_1, a_{m}\}, t_l = \{\hat{a}_{l-2}, a_{l-1}\}, \text{ for } m + 2 \leq l \leq m + n - 1, t_{m+n} = \{\hat{a}_{m+n-2}\}\} \end{split} \end{equation*}

The spectrum of $P$ follows.

 \begin{equation*} \begin{split} \mathcal{S}(P)=\left\{\left\langle r_1 = \frac{1}{m+1} - \frac{n-1}{m+1}(r_{m+n}), r_i = \frac{m}{m^2-1} - \frac{mn+1}{m^2-1}(r_{m+n}) \text{ for } 2 \leq i \leq m, \right. \right. \\  \left. \left. r_{j} = r_{m+n} \text{ for } m+   1 \leq j \leq m+n-1 \right\rangle \mid r_{m+n} \in \mathbb{Q}^+\right\} \end{split} \end{equation*} 
 
 It is easy to see from $\mathcal{S}(P)$ that if the free variable $r_{m+n}$ is set to $0$, the The least common multiple of the  vector component denominators of the resulting solution is $m^2-1$. Since the target graph order is $m+n$, if $n>m^2-m-1$ then a graph of order smaller than the target can be realized by the pot.
 
  \begin{figure}[h!]
     \centering
 \begin{tikzpicture}[transform shape, scale = 0.7]

 \node[main node] (a) at ( 9,0) {$t_1$};
 \node[main node] (b) at (7.62
, -1.90) {$t_5$};
 \node[main node](c) at (5.38
,-1.17) {$t_4$};
 \node[main node] (d) at (5.38
,1.17) {$t_3$};
 \node[main node] (e) at (7.62, 1.90
) {$t_2$};
\node[main node] (f) at ( 11,0) {$t_6$};
\node[main node] (g) at ( 13,0) {$t_7$};
\node[main node] (h) at ( 15,0) {$t_8$};
\node[main node] (i) at ( 17,0) {$t_9$};
\node[main node] (j) at ( 19,0) {$t_{10}$};
\node[main node] (k) at ( 21,0) {$t_{11}$};

\path[draw,thick,color=red,->]
(a) edge node []{} (b)
(a) edge node []{} (c)
(a) edge node []{} (d)
(a) edge node []{} (e)
(a) edge node []{} (f);

\path[draw,thick,color=blue,->]
(e) edge node []{} (b)
(e) edge node []{} (c)
(e) edge node []{} (d);

\path[draw,thick,color=orange,->]
(d) edge node []{} (b)
(d) edge node []{} (c);

\path[draw,thick,color=black,->]
(c) edge node []{} (b);

\path[draw,thick,color=purple,->]
(f) edge node []{} (g);

\path[draw,thick,color=teal,->]
(g) edge node []{} (h);

\path[draw,thick,color=green,->]
(h) edge node []{} (i);

\path[draw,thick,color=pink,->]
(i) edge node []{} (j);

\path[draw,thick,color=brown,->]
(j) edge node []{} (k);
\end{tikzpicture}

\caption{Scenario 3 labeling of $L_{5,6}$ for Example \ref{ex:1LollipopOddS3} } 
\label{fig:L5,6Scen3ex1}
\end{figure}
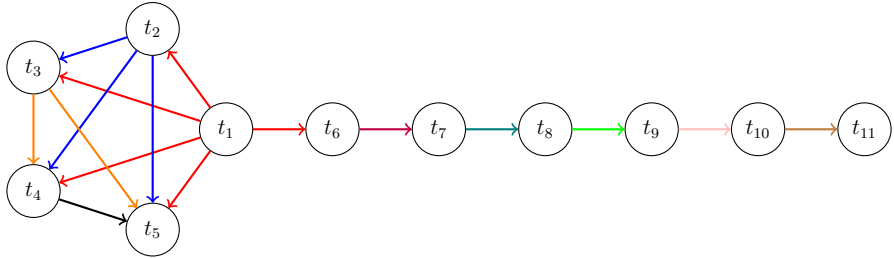
 \end{example}
 
 \begin{example} \label{ex:2LollipopOddS3}
Consider the following pot realizing $L_{m,n}$ in which a bond-edge type is used to label two edges incident with the bridging vertex, one in the path subgraph and one in the $K_m$ subgraph. An example labeling of the graph $L_{5,6}$ using this pot is shown in Figure \ref{fig:L5,6Scen3ex2}. 

\begin{equation*} \begin{split}  P = \{t_1 = \{a_1^{2}, \hat{a_2},...\hat{a_{m-1}}\}, t_i = \{\hat{a}_{2}, ..., \hat{a}_{i-1}, a_i^{m-(i-1)}\} \text{ for } 2 \leq i \leq m-1, t_m = \{\hat{a}_1,..., \hat{a}_{m-1}\}\\ t_{m+1} = \{\hat{a}_1, a_{m}\}, t_l = \{\hat{a}_{l-2}, a_{l-1}\}, \text{ for } m + 2 \leq l \leq m + n - 1, t_{m+n} = \{\hat{a}_{m+n-2}\}\} \end{split} \end{equation*}

The spectrum of $P$ follows.

 \begin{equation*} \begin{split} \mathcal{S}(P)=\left\{\left\langle r_1 = \frac{6}{3m+4} - \frac{6n-2}{3m+4}(r_{m+n}), r_i = \frac{3}{3m+4} - \frac{3n-1}{3m+4}(r_{m+n}) \text{ for } 2 \leq i \leq m-1, \right. \right. \\  \left. \left. r_m = \frac{4}{3m+4}-\frac{m+4n}{3m+4}(r_{m+n}), r_{j} = r_{m+n} \text{ for } m+1 \leq j \leq m+n-1 \right\rangle \mid r_{m+n} \in \mathbb{Q}^+\right\} \end{split} \end{equation*} 
 
  It is easy to see from $\mathcal{S}(P)$ that if the free variable $r_{m+n}$ is set to $0$, the least common multiple of the vector component denominators of the resulting solution is $3m+4$. Since the target graph order is $m+n$, if $n>2m+4$ then a graph of smaller order than the target can be realized by the pot.
  
  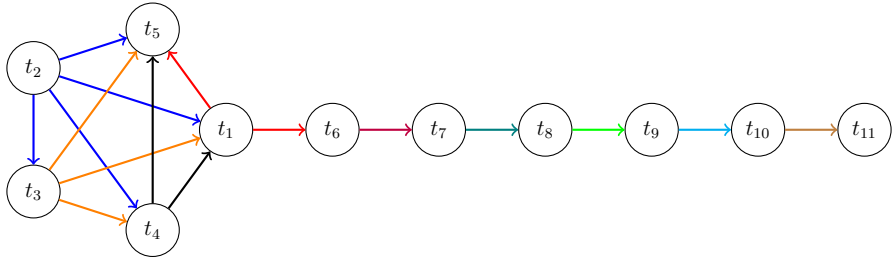
\begin{figure}[h!]
     \centering
 \begin{tikzpicture}[transform shape, scale = 0.7]

 \node[main node] (a) at ( 9,0) {$t_1$};
 \node[main node] (b) at (7.62
, -1.90) {$t_4$};
 \node[main node](c) at (5.38
,-1.17) {$t_3$};
 \node[main node] (d) at (5.38
,1.17) {$t_2$};
 \node[main node] (e) at (7.62, 1.90
) {$t_5$};
\node[main node] (f) at ( 11,0) {$t_6$};
\node[main node] (g) at ( 13,0) {$t_7$};
\node[main node] (h) at ( 15,0) {$t_8$};
\node[main node] (i) at ( 17,0) {$t_9$};
\node[main node] (j) at ( 19,0) {$t_{10}$};
\node[main node] (k) at ( 21,0) {$t_{11}$};

\path[draw,thick,color=red,->]
(a) edge node []{} (e)
(a) edge node []{} (f);

\path[draw,thick,color=blue,->]
(d) edge node []{} (e)
(d) edge node []{} (a)
(d) edge node []{} (b)
(d) edge node []{} (c);

\path[draw,thick,color=orange,->]
(c) edge node []{} (e)
(c) edge node []{} (b)
(c) edge node []{} (a);

\path[draw,thick,color=black,->]
(b) edge node []{} (a)
(b) edge node []{} (e);

\path[draw,thick,color=purple,->]
(f) edge node []{} (g);

\path[draw,thick,color=teal,->]
(g) edge node []{} (h);

\path[draw,thick,color=green,->]
(h) edge node []{} (i);

\path[draw,thick,color=cyan,->]
(i) edge node []{} (j);

\path[draw,thick,color=brown,->]
(j) edge node []{} (k);
\end{tikzpicture}

\caption{Scenario 3 labeling of $L_{5,6}$ for Example \ref{ex:2LollipopOddS3} } 
\label{fig:L5,6Scen3ex2}
\end{figure}

\end{example}

As in Scenario 2, the minimum number of tile types is significantly more straightforward for Scenario 3. The minimum number of tile types required is the worst case scenario, in which each vertex of the graph corresponds to a distinct tile type.

\begin{proposition} $T_{3}(L_{m, n}) = m + n$. \end{proposition}
\begin{proof}
By Lemma \ref{S3NoRepeatedBondEdge} $n$ distinct tile types are needed to label the vertices in the appended path. The vertices in the $K_m$ subgraph are of different degrees than the vertices in the path, and thus require tile types distinct from those used in the path. Since the $m$ vertices in the $K_m$ subgraph are all adjacent to one another, each requires a distinct tile type in order to avoid possible formation of loop edges \cite{mintiles}. This implies $T_3(L_{m,n}) \geq m+n$. The lower bound is achieved by the pot given in (\ref{eq:LollipopS3B3pot}). 
\end{proof}

\section{Tadpole Graphs}\label{Tadpolegraphsection}

\begin{definition} A \textit{tadpole graph} is a $C_m$ cycle graph connected to a $P_n$ path, through a single bridging vertex of degree $3$. \cite{TadpoleGraph} \end{definition}
 To avoid confusion with number of minimum tile types, we denote tadpole graphs as $Tad_{m,n},$ where $m$ is the number of vertices in the cycle and $n$ is the number of vertices in the extending path. Thus, the order of $Tad_{m,n}$ is $m+n$. We will often refer to the degree $3$ vertex as the \textit{bridging vertex}. Note that the bridging vertex is a cut-vertex and the path is a vertex-induced subgraph. \par 
 $B_i(C_m)$ and $T_i(C_m)$ for $i=1,2,3$ are known and these values can be found in Table \ref{table:Cycle Graph}. The cycle $C_m$ is vertex-induced subgraph of $Tad_{m,n}$ and is a maximum clique of the graph. We will occasionally refer to the values of $B_i(C_m)$ and $T_i(C_m)$ in our proofs for $Tad_{m,n}$, as some of the same labeling strategies and lower bounds still apply.
 
\begin{table}[h]
\centering
\begin{tabular}{| l | c | c| }
\hline
\renewcommand*{\arraystretch}{1.2}
 & $B_i(C_m)$ & $T_i(C_m)$ \\ \hline
Scenario 1 & $B_1(C_m) = 1$  & $T_1(C_m) = 1$ \\ \hline
Scenario 2 & $B_2(C_m) = \lceil \frac{m}{2} \rceil$  & $T_2(C_m) =  \lceil \frac{m}{2} \rceil +1$ \\ \hline
Scenario 3  & $B_3(C_m) = \lceil \frac{m}{2} \rceil$  & $T_3(C_m) =\lceil \frac{m}{2} \rceil +1$\\ \hline 
\end{tabular}
\caption{Minimum tile and bond-edge type values for the cycle graph family \cite{mintiles}} 
\label{table:Cycle Graph}
\end{table}

\subsection{Scenario 1}

Note that any tadpole graph consists of $m-1$ degree 2 vertices and a single degree 3 bridging vertex, as well as $n-1$ degree 2 vertices and a single degree 1 vertex within the path. \\

Recall that for all graphs $G$, $B_1(G)=1$ \cite{mintiles}.

\begin{proposition}
$T_1(Tad_{m,n}) = 3$.
\end{proposition}
\begin{proof}
For all values of $m$ and $n$, $av(Tad_{m,n}) = 3, ev(Tad_{m,n}) = 1,$ and $ov(Tad_{m,n}) = 2.$ By Theorem 1 of \cite{mintiles}, $3 \leq T(Tad_{m,n}) \leq 5$. The lower bound is achieved by the following pot. Example labelings of $Tad_{6,1}$ and $Tad_{5,2}$ are shown in Figure \ref{fig:Scen12}.
\begin{equation*}  P = \{t_1 = \{\hat{a}, a\}, t_2 = \{ \hat{a}, a^2\}, t_3 = \{\hat{a}\}\} \end{equation*}
\end{proof}

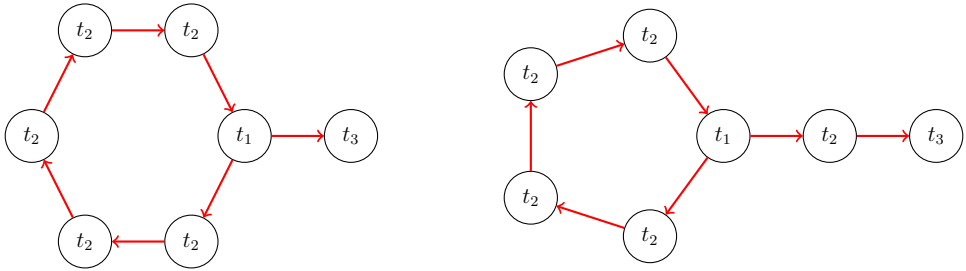
\begin{figure}[h!]
     \centering
	 \begin{tikzpicture}[transform shape, scale = 0.7]
 \node[main node] (a) at (0,0) {$t_1$};
 \node[main node] (b) at (-1,-2) {$t_2$};
 \node[main node](c) at (-3,-2) {$t_2$};
 \node[main node] (d) at (-4,0) {$t_2$};
 \node[main node] (e) at (-3,2) {$t_2$};
 \node[main node] (f) at (-1,2) {$t_2$};
 \node[main node] (g) at (2,0) {$t_3$};
		 
\path[draw,thick,color=red,->]
(a) edge node []{} (b)
(c) edge node []{} (d)
(d) edge node []{} (e)
(e) edge node []{} (f)
(f) edge node []{} (a)
(b) edge node []{} (c)
(a) edge node []{} (g);

 \node[main node] (h) at ( 9,0) {$t_1$};
 \node[main node] (i) at (7.62
, -1.90) {$t_2$};
 \node[main node](j) at (5.38
,-1.17) {$t_2$};
 \node[main node] (k) at (5.38
,1.17) {$t_2$};
 \node[main node] (l) at (7.62, 1.90
) {$t_2$};
\node[main node] (m) at ( 11,0) {$t_2$};
\node[main node] (n) at ( 13,0) {$t_3$};

\path[draw,thick,color=red,->]
(h) edge node []{} (i)
(i) edge node []{} (j)
(j) edge node []{} (k)
(k) edge node []{} (l)
(l) edge node []{} (h)
(h) edge node []{} (m)
(m) edge node []{} (n)
;
\end{tikzpicture}

\caption{Scenario 1 labeling of $Tad_{6,1}$ and $Tad_{5,2}$} 
\label{fig:Scen12}
\end{figure}

\subsection{Scenario 2}

The difficulty in determining minimum numbers of tile and bond-edge types for tadpole graphs increases dramatically with Scenario 2. Again, as with lollipop graphs, we observe differences between even and odd numbers of vertices in the non-path subgraph. Additionally, in tadpole graphs, results vary by the length of the appended path.

\begin{proposition} \label{prop:B2tadpoleshortpath}
$B_2(Tad_{m, n}) = \left\lceil \frac{m}{2} \right\rceil$ for $n \leq \left \lceil \frac{m}{2} \right\rceil$.
\end{proposition}

\begin{proof}
\label{prop:B2eventad}
 As shown in Proposition 8 of \cite{mintiles}, if there are fewer than $\left\lceil \frac{m}{2} \right\rceil$ bond-edges types used to label the edges within the $C_{m}$ subgraph, then one or more bond-edge types will be repeated at least three times and at least two of the three edges labeled with this bond-edge type will have the same labeling orientation. The resulting pot realizes two smaller graphs, a cycle and a smaller tadpole graph; an example is shown in Figure \ref{fig:smallergraphsoddtadpole}. Therefore, $B_2(Tad_{m,n}) \geq \left\lceil \frac{m}{2} \right\rceil$. \par

\begin{figure}[hb!]
     \centering
	 \begin{tikzpicture}[transform shape, scale = 0.7]
\node[main node, fill=black, minimum size = 0.3cm] (h) at ( 9,0) {};
 \node[main node, fill=black, minimum size = 0.3cm] (i) at (7.62
, -1.90) {};
 \node[main node, fill=black, minimum size = 0.3cm](j) at (5.38
,-1.17) {};
 \node[main node, fill=black, minimum size = 0.3cm] (k) at (5.38
,1.17) {};
 \node[main node, fill=black, minimum size = 0.3cm] (l) at (7.62, 1.90
) {};
\node[main node, fill=black, minimum size = 0.3cm] (m) at ( 11,0) {};
\node[main node, fill=black, minimum size = 0.3cm] (n) at ( 13,0) {};

\path[draw,thick,color=red,->]
(h) edge node []{} (i)
(j) edge node []{} (k);

\path[draw,thick,color=black]
(i) edge node []{} (j)
(k) edge node []{} (l)
(l) edge node []{} (h)
(h) edge node []{} (m)
(m) edge node []{} (n);
\end{tikzpicture} \hspace{10mm}
\begin{tikzpicture}[transform shape, scale = 0.7]
\node[main node, fill=black, minimum size = 0.3cm] (h) at ( 9,0) {};
 \node[main node, fill=black, minimum size = 0.3cm] (i) at (9
, -2) {};
 \node[main node, fill=black, minimum size = 0.3cm](j) at (5
,-2) {};
 \node[main node, fill=black, minimum size = 0.3cm] (k) at (5
,0) {};
 \node[main node, fill=black, minimum size = 0.3cm] (l) at (7, 1.90
) {};
\node[main node, fill=black, minimum size = 0.3cm] (m) at ( 11,0) {};
\node[main node, fill=black, minimum size = 0.3cm] (n) at ( 13,0) {};

\path[draw,thick,color=red,->]
(j) edge [bend left] node []{} (i)
(h) edge node []{} (k);

\path[draw,thick,color=black]
(i) edge [bend left] node []{} (j)
(k) edge node []{} (l)
(l) edge node []{} (h)
(m) edge node []{} (n)
(h) edge node []{} (m);
\end{tikzpicture}
\caption{Smaller graphs formed from a $Tad_{5,2}$ labeling} 
\label{fig:smallergraphsoddtadpole}
\end{figure}
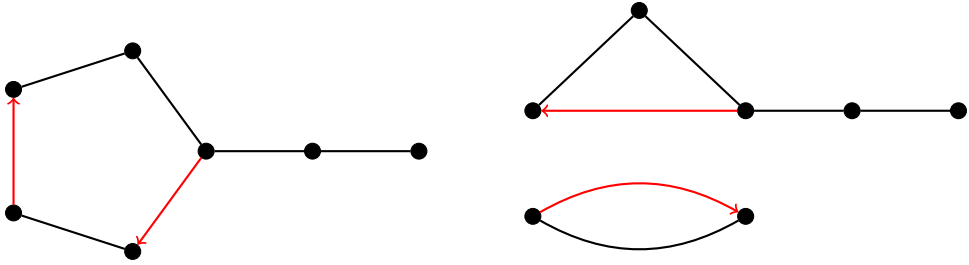

The lower bound is achieved by the following pots for $m$ even and odd. Example labelings of $Tad_{5,2}$ and $Tad_{6,2}$ are shown in Figure \ref{fig:tadpoleS2shortpath}.

\begin{equation}\label{eq:TadpoleS2evenshortpath} \begin{split}  P_{(2k,n)} = \{t_1 = \{a_1^2, a_{k - n + 1}\}, t_i = \{\hat{a}_{i - 1}, a_i\} \text{ for } 2 \leq i \leq  k,
t_{k+ 1} = \{\hat{a}_k^2\},  t_{k + 2} = \{\hat{a}_{k}\}\} \end{split} \end{equation}

\begin{equation}\label{eq:TadpoleS2oddshortpath} \begin{split}  P_{(2k+1,n)} = \{t_1 = \{a_1^2, a_{k-n+2}\}, t_i = \{\hat{a}_{i - 1}, a_i\} \text{ for } 2 \leq i \leq k, t_{k+1} = \{\hat{a}_k, \hat{a}_{k+1}\},\\
t_{k+2} = \{\hat{a}_k, a_{k+1}\}, t_{k+3} = \{\hat{a}_k\}\} \end{split} \end{equation}

The construction matrices and spectrums of $P_{(2k,n)}$ and $P_{(2k+1,n)}$ follow. In each matrix, the upward arrow in the first column denotes that the position of the `1' below the arrow moves upward as the value of $n$ increases, since one arm of $t_1$ is dependent on the value of $n$. The `1' moves from row $\lceil \frac{m}{2} \rceil$ ($n=1$) to row 2 ($n=\lceil \frac{m}{2} \rceil -1$), and if $n=\lceil \frac{m}{2} \rceil$ then the `2' in matrix entry $(1,1)$ will become a `3'.  Even so, it is still possible to determine general solution sets.

\begin{equation*}M(P_{(2k,n)}) = \begin{bmatrix} 
2 & -1 & 0 & 0 &\cdots & 0 & 0 & 0  \\ 
0 & 1 & -1 & 0 & \cdots & 0 & 0 & 0 \\ 
\vdots & 0 & \ddots & \ddots & \ddots & 0 & 0 & 0 \\
\uparrow  & 0 & \cdots & 1 & -1 & 0 & 0 & 0 \\
1 & 0 & \cdots & 0  & -1 & -2 & -1 & 0  \\
1 & 1 &  \cdots & 1 & 1 & 1 & 1 & 1 \end{bmatrix}\end{equation*}

\begin{equation*} \begin{split} \mathcal{S}(P_{(2k,n)}) = \left\{\left\langle r_1 = \frac{2}{2(2k +  n)- 1} - \frac{1}{2(2k +  n)- 1}r_{k+2},\right. \right. \\ r_i = \frac{4}{2(2k +  n)- 1} - \frac{2}{2(2k +  n)- 1}r_{k+2}
\text{ for } 2 \leq i \leq  k -n + 1, \\ r_j \frac{6}{2(2k +  n)- 1} - \frac{3}{2(2k +  n)- 1}r_{k+2} \text{ for } k - n + 2 \leq j \leq k,\\ r_{k + 1} = \frac{3}{2(2k +  n)- 1} - \frac{2k + n + 1}{2(2k +  n)- 1}r_{k + 2}, \left. \left. r_{k + 2}  \right\rangle \mid r_{k + 2} \in \mathbb{Q}^+\right\} \end{split} \end{equation*}

\begin{equation} M(P_{(2k+1,n)}) = \begin{bmatrix} 
2 & -1 & 0 & 0 & 0 &\cdots  & 0 & 0 & 0  \\ 
0 & 1  & -1 & 0 & 0 & \cdots & 0 & 0 & 0 \\ 
\vdots & \ddots & \ddots & \ddots & \ddots & \ddots & 0 & 0 & 0 \\
\uparrow & 0 & 0 & 1  & -1 & 0 & 0 & 0 & 0  \\ 
1 & 0 & 0 & 0 & 1 & -1  & -1 & -1 & 0 \\
0 & 0 & 0 & 0 & 0 & -1 & 1  & 0 & 0  \\
1 &  1 & 1 & \cdots & 1 & 1 & 1 & 1 & 1 \end{bmatrix} \end{equation}

\begin{equation*} \begin{split} \mathcal{S}(P_{(2k+1,n)}) = \left\{\left\langle r_1 = \frac{1}{2k+n+1}, r_i = \frac{2}{2k+n+1} \text{ for }  2 \leq i \leq k+1 - n, r_j = \frac{3}{2k+1 + n} \right. \right. \\ \left. \left.\text{ for } k - n + 2 \leq j \leq k, 
r_{k+1}, r_{k+2} = \frac{3}{2(2k+n+1)} - \frac{1}{2}r_{k+3}, r_{k+3} \right\rangle \mid  r_{k+3} \in \mathbb{Q}^+\right\} \end{split} \end{equation*}

The least common multiple of the vector component denominators in $\mathcal{S}(P_{(2k+1,n)})$ is at least $2k+n+1$,  and therefore no graphs of smaller order than $Tad_{2k+1,n}$ can be realized by the pot.\par 
On the other hand, $\mathcal{S}(P_{(2k,n)})$ has one degree of freedom and no clear least common multiple of the vector component denominators. Therefore, it is not immediately helpful in determining that no smaller graphs can be realized. Instead, we use a combination of $\mathcal{S}(P_{(2k,n)})$ and the ability of various tiles to bond together to argue that nothing smaller will be realized. If $r_1 = 0,$ then $r_{k + 2} = 2,$ and $r_{k + 1} =  -1.$ Tile proportions outside of the allowable range $[0,1]$ signify that a complete complex can not form if $r_1 = 0$. Consequently, $t_1$ must be included. A single $t_1$ tile must bond to two $t_2$ tiles, as that is the only tile with an $\hat{a_1}$ arm. Similarly, a $t_3$ tile must bond to each $t_2$ tile's $a_2$ arm. This process continues until two $t_{k}$ tiles are bonded to the complex, each with an extending $a_{k}$ arm. In addition, the $a_{k-n+1}$ arm of tile $t_1$ must bond to tile $t_{k-n+2}$. Once again, this process will continue until tile $t_{k}$ which has an extending $a_{k}$ arm. This complex is shown in Figure \ref{fig:B2evenTad}.

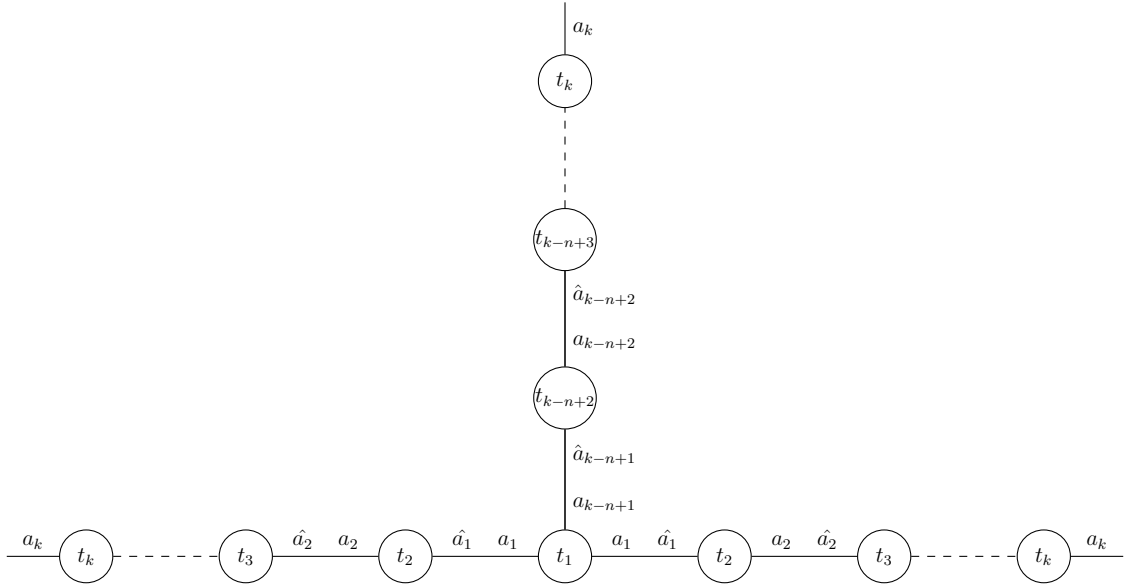
\begin{figure}[h!]
\centering
	 \begin{tikzpicture}[transform shape, scale = 0.7]
 \node[main node] (a) at (-9,0) {$t_k$};
 \node[main node] (b) at (-6,0) {$t_{3}$};
 \node[main node] (c) at (-3,0) {$t_{2}$};
  \node[main node] (k) at (0,3) {$t_{k-n+2}$};
 \node[main node] (l) at (0,6) {$t_{k-n+3}$};
 \node[main node] (o) at (0,9) {$t_{k}$};
  \node[main node] (d) at (0,0) {$t_1$};
\node[main node] (g) at (3,0) {$t_2$};
\node[main node] (h) at (6,0) {$t_3$};
\node[main node] (i) at (9,0)  {$t_{k}$};
 \node[main node, fill=white, draw=white, minimum size = 0cm] (j) at (0,10.5) {};
 \node[main node, fill=white, draw=white, minimum size = 0cm] (e) at (-10.5,0) {};
 \node[main node, fill=white, draw=white, minimum size = 0cm] (f) at (10.5,0){};
		 
\path[draw]
(d) edge node [near start, right]{$a_{k-n+1}$} (k)
(k) edge node [near start, right]{$\hat{a}_{k-n+1}$} (d)
(l) edge node [near start, right]{$\hat{a}_{k-n+2}$} (k)
(k) edge node [near start, right]{$a_{k-n+2}$} (l)
(b) edge node [above]{$\hat{a_2}$ \hspace{2mm} $a_2$} (c)
(c) edge node [above]{$\hat{a_1}$ \hspace{2mm} $a_1$} (d)
(d) edge node [above]{$a_1$ \hspace{2mm} $\hat{a_1}$} (g)
(g) edge node [above]{$a_2$ \hspace{2mm} $\hat{a_2}$} (h)
(i) edge node [above]{$a_k$} (f)
(o) edge node [right]{$a_k$} (j)
(a) edge node [above]{$a_k$} (e);

\path[draw, dashed]
(a) edge node []{} (b)
(h) edge node []{} (i)
(l) edge node []{} (o);
\end{tikzpicture}

\caption{Complex formed in proof of Proposition \ref{prop:B2eventad}} 
\label{fig:B2evenTad}
\end{figure}
At this point the complex is of size $2k - 2 + n$ and has exactly three un-bonded extending arms, all of which are labeled with $a_k$. These arms cannot bond to one another and there are no tiles in $P_{(2k,n)}$ with three $\hat{a_k}$ arms, so at least two tiles are needed. Therefore, any complete complex will be of size at least $2k+n$, the target graph order.
\end{proof}

\begin{figure}[h!]
     \centering
\begin{tikzpicture}[transform shape, scale = 0.7]
 \node[main node] (i) at (-6,0) {$t_1$};
 \node[main node] (j) at (-7.38
, -1.90) {$t_2$};
 \node[main node](k) at (-9.62
,-1.17) {$t_4$};
 \node[main node] (l) at (-9.62
,1.17) {$t_3$};
 \node[main node] (m) at (-7.38, 1.90
) {$t_2$};
\node[main node] (n) at (-4,0) {$t_3$};
\node[main node] (o) at (-2,0) {$t_5$};

\path[draw,thick,color=red,->]
(i) edge node []{} (m)
(i) edge node []{} (j);

\path[draw,thick,color=blue,->]
(m) edge node []{} (l)
(j) edge node []{} (k)
(i) edge node []{} (n);

\path[draw,thick,color=green,->]
(n) edge node []{} (o)
(l) edge node []{} (k);

 \node[main node] (a) at (5,0) {$t_1$};
 \node[main node] (b) at (4,-2) {$t_2$};
 \node[main node](c) at (2,-2) {$t_3$};
 \node[main node] (d) at (1,0) {$t_4$};
 \node[main node] (e) at (2,2) {$t_3$};
 \node[main node] (f) at (4,2) {$t_2$};
 \node[main node] (g) at (7,0) {$t_3$};
  \node[main node] (h) at (9,0) {$t_5$};
		 
\path[draw,thick,color=red,->]
(a) edge node []{} (f)
(a) edge node []{} (b);

\path[draw,thick,color=blue,->]
(f) edge node []{} (e)
(b) edge node []{} (c)
(a) edge node []{} (g);

\path[draw,thick,color=green,->]
(c) edge node []{} (d)
(g) edge node []{} (h)
(e) edge node []{} (d);

\end{tikzpicture}

\caption{Scenario 2 labelings of $Tad_{5,2}$ and $Tad_{6,2}$} 
\label{fig:tadpoleS2shortpath}
\end{figure}
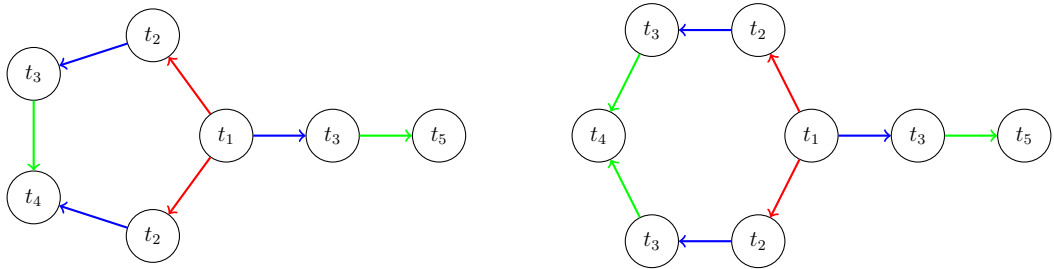

\begin{proposition}\label{prop:T2tadpoleshortpath}
$T_2(Tad_{m,n})=\left\lceil\frac{m}{2}\right\rceil+2$ for $n \leq \left\lceil \frac{m}{2}\right\rceil$.
\end{proposition}
\begin{proof}
Suppose $P=\{t_1,t_2,...,t_p,t_{p+1},...,t_{q-1},t_q\}$ is a pot realizing $Tad_{m,n}$, where $t_1,...,t_p$ are tiles used to label the cycle and $t_p$ is the tile used to label the bridging vertex. Let $P'=\{t_1,...,t_{p-1},t_p'\}$, where $t_p'$ is a two-armed tile formed from removing the arm of $t_p$ used to label the half-edge extending towards the path. It is clear that if $P$ realizes the graph $Tad_{m,n}$ and no smaller graphs, then $P'$ must realize the graph $C_{m}$ and no smaller graphs; if $P'$ realizes a graph $H$ smaller than $C_{m}$, then $P$ would realize the graph formed by attaching the path to $H$ via the third arm of tile $t_p$, which would be smaller than $Tad_{m,n}$. By the argument presented in Proposition \ref{prop:B2tadpoleshortpath}, $P'$ must consist of at least $\left\lceil \frac{m}{2}\right\rceil$ bond-edge types. Since $T_2(G) \geq B_2(G) + 1$ for any graph $G$ \cite{mintiles}, $P'$ must consist of at least $\left\lceil \frac{m}{2} \right\rceil +1$ distinct tile types, so $P$ must consist of at least $\left\lceil \frac{m}{2} \right\rceil +1$ distinct tile types used to label degree 2 vertices in the cycle and, of course, a distinct tile type to label the degree 3 vertex. The vertex of degree 1 in $Tad_{m,n}$ requires an additional tile type, so $T_2(Tad_{m,n}) \geq \left\lceil \frac{m}{2} \right\rceil +2$. \par
The pots given in (\ref{eq:TadpoleS2evenshortpath}) and (\ref{eq:TadpoleS2oddshortpath}) achieve the lower bound.
\end{proof}

\begin{proposition}\label{prop:tadpoleS2mediumpath}
$B_2(Tad_{m, n}) = n$ and $T_2(Tad_{m, n}) = n + 2$ for $\lceil \frac{m}{2} \rceil \leq n \leq m$.
\end{proposition}
\begin{proof} 
By Lemma \ref{S2NoRepeatedBondEdge}, $B_2(Tad_{m,n}) \geq n$ and $n$ distinct tile types are required to label the vertices of the path. 
There are $n-1$ two-armed tile tiles, and the arms have different bond-edge types. 
To prevent new two-armed tile types from being necessary in the cycle graph, since $n \leq m$, at least one two-armed tile type used in the path must be used to label multiple vertices in the cycle.  
A bond-edge type cannot be repeated more than twice in the cycle. Therefore, a two-armed tile in the path can be used only twice in the cycle. Furthermore, the edges created must be of opposite orientation as regarded when moving around the cycle. In this way, we can repeat up to $\lceil \frac{m}{2} \rceil - 1$ of the $n-1$ two-armed tile types from the path in the cycle. One additional two-armed tile type is required to label one of the degree 2 vertices in the cycle due to opposite orientations of repeated bond-edges. Finally, a three-armed tile type is needed to label the bridging vertex. Thus, $T_2(Tad_{m,n}) \geq n+2.$ The following pots for $m$ even and odd achieve the lower bounds. Example labelings of $Tad_{5,4}$ and $Tad_{6,4}$ are shown in Figure \ref{fig:tadpoleS2mediumpath}.

\begin{equation}\label{eq:T2evenlongpathpot} \begin{split}  P_{(2k,n)} = \{t_1 = \{a_1^3\}, t_i = \{\hat{a}_{i - 1}, a_i\} \text{ for } 2 \leq i \leq n, t_{n+1} = \{\hat{a}_{n}\},
t_{n+2} = \{\hat{a}_{k}^2\} \end{split} \end{equation}

\begin{equation}\label{eq:tadpoleoddlongpathpot} \begin{split}  P_{(2k+1,n)} = \{t_1 = \{a_1^3\}, t_i = \{\hat{a}_{i - 1}, a_i\} \text{ for } 2 \leq i \leq n, t_{n+1} = \{\hat{a}_{n}\},
  t_{n+2} = \{\hat{a}_{k}, \hat{a}_{k+1}\} \end{split} \end{equation}

The construction matrices and spectrums of $P_{(2k,n)}$ and $P_{(2k+1,n)}$ follow.

\begin{equation*} M(P_{(2k,n)}) = \begin{bmatrix} 
3 & -1 & 0 & \cdots & & & & & \cdots & 0 & 0 \\ 
0 & 1 & -1 & 0 & \cdots & & & & \cdots & 0 & 0 \\ 
\vdots & \ddots & \ddots & \ddots & \ddots & & & &  & \vdots & \vdots\\
 & & 0 & 1 & -1 & 0 & \cdots & & \cdots & 0 & 0 \\
 & & & 0 & 1 & -1 & 0 & \cdots & 0 & -2 & 0 \\
 & & & & 0 & 1 & -1 & 0 & \cdots & 0 & 0 \\ 
\vdots & & & & & \ddots & \ddots & \ddots & \ddots & \vdots & \vdots \\
0 & \cdots & & & & \cdots & 0 & 1 & -1 & 0 & 0 \\
1 & 1 & \cdots & & & & & \cdots & 1 & 1 & 1 \end{bmatrix} \end{equation*}

\begin{equation*} \begin{split} \mathcal{S}(P_{(2k,n)}) = \left\{\left\langle r_1 = \frac{1}{3n+1} - \frac{2k-2n-1}{3n+1}r_{n + 2}, r_i = \frac{3}{3n+1} - \frac{6k-6n-3}{3n+1}r_{n + 2} \right.\right.\\ \text{ for } i = 2, \cdots , k, r_j = \frac{3}{3n+1}-\frac{6k-1}{3n+1}r_{n+2} \text{ for } k+1 \leq j \leq n+1, \\\left. \left.
r_{n+2}\right\rangle \mid r_{n + 2} \in \mathbb{Q}^+ \right\}\end{split} \end{equation*}

\begin{equation*} M(P_{(2k+1,n)}) = \begin{bmatrix} 
3 & -1 & 0 & \cdots & & & & & & \cdots & 0 & 0 & 0 \\ 
0 & 1 & -1 & 0 & \cdots & & & & & \cdots & 0 & 0 & 0 \\
\vdots & \ddots & \ddots & \ddots & \ddots & & & & & & \vdots & \vdots & \vdots \\
 & & 0 & 1 & -1 & 0 & \cdots & & & \cdots & 0 & 0 & 0 \\
 & & & 0 & 1 & -1 & 0 & \cdots & & \cdots & 0 & -1 & 0 \\
 & & & & 0 & 1 & -1 & 0 & \cdots & & 0 & -1 & 0 \\
 & & & & & 0 & 1 & -1 & 0 & \cdots & 0 & 0 & 0 \\
\vdots & & & & & & \ddots & \ddots & \ddots & \ddots & \vdots & \vdots & \vdots \\
0 & \cdots & & & & & \cdots & 0 & 1 & -1 & 0 & 0 & 0 \\
1 & 1 & \cdots & & & & & & & \cdots & 1 & 1 & 1 \end{bmatrix} \end{equation*}

\begin{equation*} \begin{split} \mathcal{S}(P_{(2k+1,n)}) = \left\{\left\langle r_1 = \frac{1}{3n+1} - \frac{2k-2n}{3n+1}r_{n+2}, r_i = \frac{3}{3n+1} - \frac{6k-6n}{3n+1}r_{n + 2} \text{ for } 2 \leq i \leq k, \right.\right.\\ r_{k+1} = \frac{3}{3n+1}-\frac{6k-3n+1}{3n+1}r_{n+2}, r_j = \frac{3}{3n+1} -\frac{6k+2}{3n+1}r_{n+2} \\\left. \left. \text{ for } k+2 \leq j \leq n+1,   r_{n+2} \right\rangle \mid r_{n + 2} \in \mathbb{Q}^+\right\}\end{split} \end{equation*}

The spectrums of these pots each have one degree of freedom and no clear least common multiple of the vector component denominators, so they are not immediately helpful in determining that no smaller graphs can be realized. Once again we use a combination of the spectrum and the ability of various tile arms to bond together to show that nothing smaller will be realized. \par 
For both $\mathcal{S}(P_{(2k,n)})$ and $\mathcal{S}(P_{2k+1,n}))$, if $r_1 = 0,$ then $r_{n+2} <0$, since $n \geq \lceil \frac{m}{2} \rceil$. Tile proportions outside of the allowable range $[0,1]$ signify that a complete complex cannot be realized if $r_1=0$, so $t_1$ must be included. In both pots, a single $t_1$ tile must bond to three $t_2$ tiles, as that is the only tile with an $\hat{a}_1$ arm. Similarly, tiles $t_3,...,t_k$ must bond in sequence to each $t_2$ tile, forming an incomplete complex of size $3k-2$. This complex, as shown in Figure \ref{fig:tadpoleS2mediumpath},
consists of a central $t_1$ tile and three emanating paths, each ending in an un-bonded $a_k$ arm. Tiles $t_{k+1}$ or $t_{n+2}$ can bond to each of these three arms. 

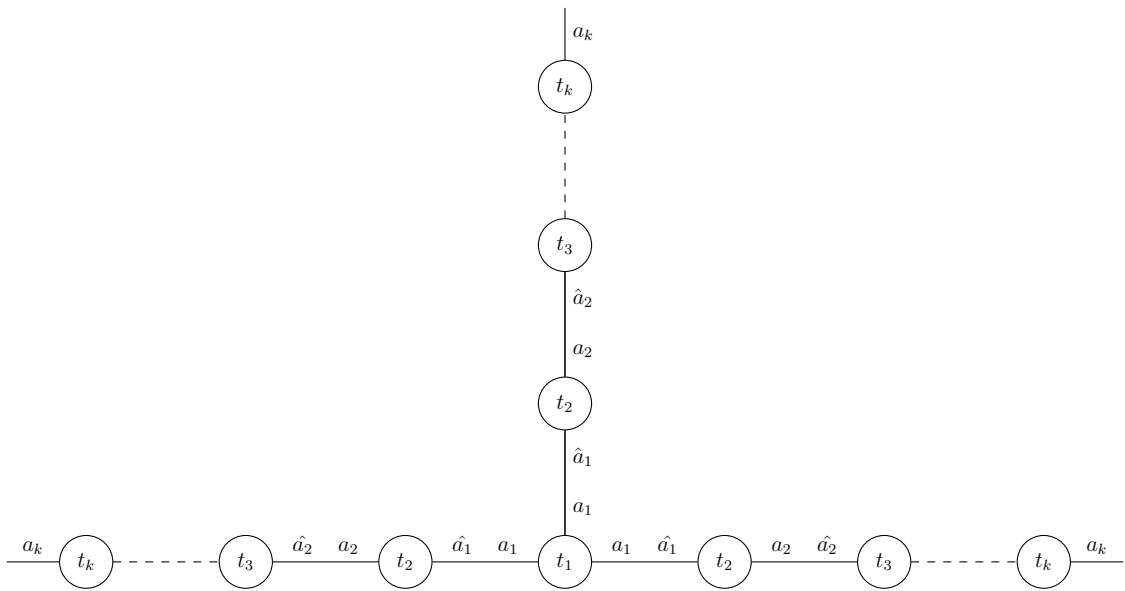
\begin{figure}[h!]
\centering
	 \begin{tikzpicture}[transform shape, scale = 0.7]
 \node[main node] (a) at (-9,0) {$t_k$};
 \node[main node] (b) at (-6,0) {$t_{3}$};
 \node[main node] (c) at (-3,0) {$t_{2}$};
  \node[main node] (k) at (0,3) {$t_{2}$};
 \node[main node] (l) at (0,6) {$t_{3}$};
 \node[main node] (o) at (0,9) {$t_{k}$};
  \node[main node] (d) at (0,0) {$t_1$};
\node[main node] (g) at (3,0) {$t_2$};
\node[main node] (h) at (6,0) {$t_3$};
\node[main node] (i) at (9,0)  {$t_{k}$};
 \node[main node, fill=white, draw=white, minimum size = 0cm] (j) at (0,10.5) {};
 \node[main node, fill=white, draw=white, minimum size = 0cm] (e) at (-10.5,0) {};
 \node[main node, fill=white, draw=white, minimum size = 0cm] (f) at (10.5,0){};
		 
\path[draw]
(d) edge node [near start, right]{$a_{1}$} (k)
(k) edge node [near start, right]{$\hat{a}_{1}$} (d)
(l) edge node [near start, right]{$\hat{a}_{2}$} (k)
(k) edge node [near start, right]{$a_{2}$} (l)
(b) edge node [above]{$\hat{a_2}$ \hspace{2mm} $a_2$} (c)
(c) edge node [above]{$\hat{a_1}$ \hspace{2mm} $a_1$} (d)
(d) edge node [above]{$a_1$ \hspace{2mm} $\hat{a_1}$} (g)
(g) edge node [above]{$a_2$ \hspace{2mm} $\hat{a_2}$} (h)
(i) edge node [above]{$a_k$} (f)
(o) edge node [right]{$a_k$} (j)
(a) edge node [above]{$a_k$} (e);

\path[draw, dashed]
(a) edge node []{} (b)
(h) edge node []{} (i)
(l) edge node []{} (o);
\end{tikzpicture}

\caption{Complex formed in proof of Proposition \ref{prop:tadpoleS2mediumpath}.} 
\label{fig:tadpoles2mediumpath}
\end{figure}

In $P_{(2k,n)}$, to minimize the number of additional tiles, $t_{n+2}$ must bond to two of the extending $a_k$ arms, bringing the complex to size $3k-1$. The other possibilities for tile attachments ($t_{k+1}$ bonding to two or three of the $a_k$ arms) leads to extending paths beyond the length of the target graph. For the final extending $a_k$ arm, there are two possible options for tile attachment. If tile $t_{n+2}$ bonds to the $a_k$ arm, then a $t_k$ tile must bond to the $t_{n+2}$ tile as that is the only tile with an $a_k$ arm. Similarly, a $t_{k-1}$ tile must bond to $t_k$ as that is the only tile with an $a_{k-1}$ arm. This process continues until $t_1$ bonds to $t_2$, as that is the only tile with an $a_1$ arm. The resulting incomplete complex is size $4k.$ Note that $4k \geq 2k+n$ given the restrictions on $n,$ so this option only realizes graphs of greater order than $Tad_{2k,n}.$ If $t_{k+1}$ bonds to the final $a_k$ arm, then the tile $t_{k+2}$ must bond to the $t_{k+1}$ tile as that is the only tile with an $\hat{a}_{k+1}$ arm. Similarly, a $t_{k+3}$ tile must bond to $t_{k+2}$. This process continues until it reaches tile $t_{n}.$ At this point, the only possibility is for the $t_{n+1}$ tile to bond to $t_n$ as that is the only tile with a $\hat{a}_n$ arm. The resulting complex is size $2k+n$ with no un-bonded arms remaining. In fact, this realizes the target graph $Tad_{2k,n}$.

In $P_{(2k+1,n)}$, to minimize the number of additional tiles, $t_{k+1}$ must bond to one extending $a_k$ while $t_{n+2}$ must bond to another extending $a_k.$ This allows for $t_{n+2}$ and $t_{k+1}$ to bond together via the bond-edge type $a_k$. The other tile attachment possibilities form complexes of greater size than the target graph. There are two possible options for the extending $a_k$ arm ($t_{k+1}$ or $t_{n+2}$). Each option leads to complexes which are analogous to those described for $P_{(2k,n)}.$ The resulting complexes are of size greater than or equal to $2k+n+1$. The complex extended with the $t_{k+1}$ tile realizes the target graph  $Tad_{2k+1,n}$.

Therefore, graphs realized by $P_{(2k,n)}$ or $P_{(2k+1,n)}$ are the target graphs or graphs of strictly greater order than the target graphs. 
\end{proof}

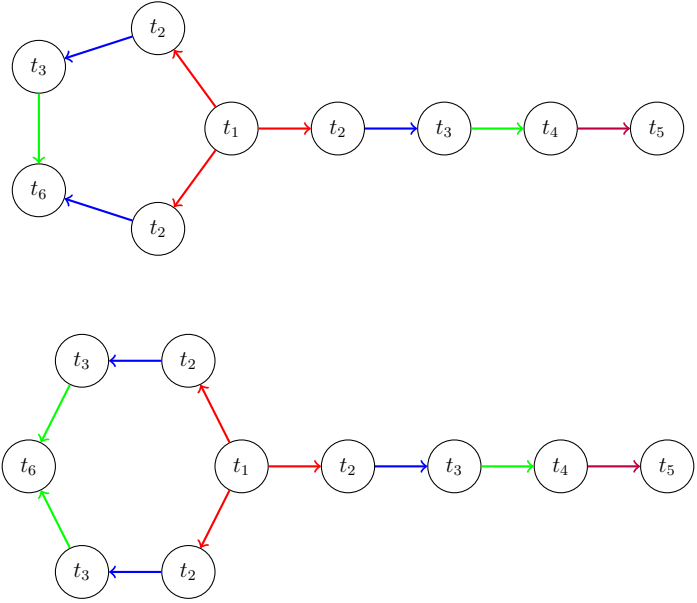
\begin{figure}[h!]
\label{fig: Scen2tad5,4}
     \centering
\begin{tikzpicture}[transform shape, scale = 0.7]
 \node[main node] (i) at (-6,0) {$t_1$};
 \node[main node] (j) at (-7.38
, -1.90) {$t_2$};
 \node[main node](k) at (-9.62
,-1.17) {$t_6$};
 \node[main node] (l) at (-9.62
,1.17) {$t_3$};
 \node[main node] (m) at (-7.38, 1.90
) {$t_2$};
\node[main node] (n) at (-4,0) {$t_2$};
\node[main node] (o) at (-2,0) {$t_3$};
\node[main node] (p) at (0,0) {$t_4$};
\node[main node] (q) at (2,0) {$t_5$};

\path[draw,thick,color=red,->]
(i) edge node []{} (m)
(i) edge node []{} (j)
(i) edge node []{} (n);

\path[draw,thick,color=blue,->]
(m) edge node []{} (l)
(n) edge node []{} (o)
(j) edge node []{} (k);

\path[draw,thick,color=green,->]
(o) edge node []{} (p)
(l) edge node []{} (k);

\path[draw,thick,color=purple,->]
(p) edge node []{} (q);

\end{tikzpicture}

\vspace{1cm}

\begin{tikzpicture}[transform shape, scale = 0.7]
 \node[main node] (a) at (5,0) {$t_1$};
 \node[main node] (b) at (4,-2) {$t_2$};
 \node[main node](c) at (2,-2) {$t_3$};
 \node[main node] (d) at (1,0) {$t_6$};
 \node[main node] (e) at (2,2) {$t_3$};
 \node[main node] (f) at (4,2) {$t_2$};
 \node[main node] (g) at (7,0) {$t_2$};
  \node[main node] (h) at (9,0) {$t_3$};
  \node[main node] (i) at (11,0) {$t_4$};
    \node[main node] (j) at (13,0) {$t_5$};
		 
\path[draw,thick,color=red,->]
(a) edge node []{} (f)
(a) edge node []{} (b)
(a) edge node []{} (g);

\path[draw,thick,color=blue,->]
(f) edge node []{} (e)
(g) edge node []{} (h)
(b) edge node []{} (c);

\path[draw,thick,color=green,->]
(c) edge node []{} (d)
(h) edge node []{} (i)
(e) edge node []{} (d);

\path[draw,thick,color=purple,->]
(i) edge node []{} (j);

\end{tikzpicture}

\caption{Scenario 2 labelings of $Tad_{5,4}$ and $Tad_{6,4}$} 
\label{fig:tadpoleS2mediumpath}
\end{figure}

\begin{remark} For graphs $Tad_{m,n}$ in which $n = \lceil \frac{m}{2} \rceil$, $B_2(Tad_{m,n})$ and $T_2(Tad_{m,n})$ are included in Propositions \ref{prop:B2tadpoleshortpath} and \ref{prop:T2tadpoleshortpath} as well in Proposition \ref{prop:tadpoleS2mediumpath}. For a graph with path length $\lceil \frac{m}{2} \rceil$, the pots given in these propositions are identical. For other path lengths, the pots are not identical, but are still quite similar. Due to different connectivity of tile types, we have chosen to keep these cases as separate propositions. \end{remark}

\begin{proposition}\label{prop:tadpoleS2longpath}
$\lceil \frac{n-m+1}{2}\rceil +m-1 \leq B_2(Tad_{m, n}) \leq n$ and $\lceil \frac{n-m+1}{2}\rceil +m+1 \leq T_2(Tad_{m, n}) \leq n + 1$ for $n > m$.
\end{proposition}
\begin{proof}
By Lemma \ref{S2RepeatedBondEdge}, $B_2(Tad_{m,n}) \geq \lceil \frac{n-m+1}{2}\rceil +m-1$ and $T_2(Tad_{m,n}) \geq \lceil \frac{n-m+1}{2}\rceil +m$, with $\lceil \frac{n-m+1}{2}\rceil +m$ distinct tile types required to label the vertices of the path. Note that the degree 3 bridging vertex will require an additional distinct tile type, so $T_2(Tad_{m,n}) \geq \lceil \frac{n-m+1}{2}\rceil +m+1$. The following pot realizes $Tad_{m,n}$ and no smaller graphs, showing $B_2(Tad_{m, n}) \leq n$ and $T_2(Tad_{m, n}) \leq n + 1$. An example labeling of $Tad_{5,7}$ is shown in Figure \ref{fig:tadpoleS2longpath}.

\begin{equation} P = \{t_1=\{a_1^2,\hat{a}_m\}, t_i = \{\hat{a}_{i-1},a_{i}\} \text{ for } 2 \leq i \leq n, t_{n+1}=\{\hat{a}_n\}\} 
\end{equation}

The construction matrix and spectrum of $P$ follows.

\begin{equation*} M(P) = \begin{bmatrix} 
2 & -1 & 0 & \cdots & & & & & \cdots & 0 \\ 
0 & 1 & -1 & 0 & \cdots & & & & \cdots & 0 \\
\vdots & \ddots & \ddots & \ddots & \ddots & & & & & \vdots \\
0 & & 0 & 1 & -1 & 0 & \cdots & & \cdots & 0 \\
-1 & 0 & \cdots & 0 & 1 & -1 & 0 & & \cdots & 0 \\
0 & & & & 0 & 1 & -1 & 0 & \cdots & 0 \\ 
\vdots & & & & & \ddots & \ddots  & \ddots & \ddots &  \vdots \\
0 & \cdots & & & & \cdots & 0 & 1 & -1 & 0  \\
1 & 1 & \cdots & & & & & \cdots & 1 & 1 \end{bmatrix} \end{equation*}

\begin{equation} \mathcal{S}(P) = \left\{ \left \langle r_1 = \frac{1}{m+n}, r_i = \frac{2}{m+n} \text{ for } 2 \leq i \leq m, r_j = \frac{1}{m+n} \text{ for } m+1 \leq j \leq n+1 \right \rangle \right\}
\end{equation}

The least common multiple of the vector component denominators in $\mathcal{S}(P)$ is $m+n$. Therefore, no graphs of smaller order than $Tad_{m,n}$ are realized by $P$.
\end{proof}

\begin{figure}[h!]
     \centering
\begin{tikzpicture}[transform shape, scale = 0.7]
 \node[main node] (i) at (-6,0) {$t_1$};
 \node[main node] (j) at (-7.38
, -1.90) {$t_2$};
 \node[main node](k) at (-9.62
,-1.17) {$t_4$};
 \node[main node] (l) at (-9.62
,1.17) {$t_3$};
 \node[main node] (m) at (-7.38, 1.90
) {$t_5$};
\node[main node] (n) at (-4,0) {$t_2$};
\node[main node] (o) at (-2,0) {$t_3$};
\node[main node] (p) at (0,0) {$t_4$};
\node[main node] (q) at (2,0) {$t_5$};
\node[main node] (r) at (4,0) {$t_6$};
\node[main node] (s) at (6,0) {$t_7$};
\node[main node] (t) at (8,0) {$t_8$};

\path[draw,thick,color=red,->]
(i) edge node []{} (m)
(i) edge node []{} (n);

\path[draw,thick,color=blue,->]
(m) edge node []{} (l)
(n) edge node []{} (o);

\path[draw,thick,color=green,->]
(o) edge node []{} (p)
(l) edge node []{} (k);

\path[draw,thick,color=purple,->]
(k) edge node []{} (j)
(p) edge node []{} (q);

\path[draw,thick,color=teal,->]
(j) edge node []{} (i)
(q) edge node []{} (r);

\path[draw,thick,color=brown,->]
(r) edge node []{} (s);

\path[draw,thick,color=pink,->]
(s) edge node []{} (t);

\end{tikzpicture}

\caption{Scenario 2 labeling of $Tad_{5,7}$} 
\label{fig:tadpoleS2longpath}
\end{figure}
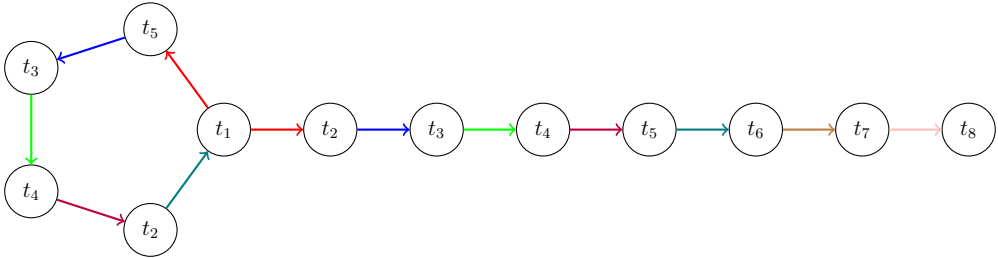

\subsection{Scenario 3}
 
The minimum numbers of tile and bond-edge types required in Scenario 3 for tadpole graphs for which the path contains fewer than half the number of vertices that are in the cycle differ from the Scenario 2 values, while tadpole graphs with longer paths maintain the same values in Scenario 3 as in Scenario 2. Scenario 3 for shorter path tadpole graphs is a nuanced problem. We first provide two lemmas to assist.

\begin{lemma}\label{lemma1:tadpoleS3}
In order to satisfy the requirements of Scenario 3 for $Tad_{m,n}$, if a bond-edge type is used to label both an edge in the cycle and an edge in the path then the number of edges between the bridging vertex and the appearance of the bond-edge type in the cycle must equal the number of edges between the bridging vertex and the appearance of the bond-edge type in the path.
\end{lemma}
\begin{proof}
 Let $A_c$ be an edge in the cycle labeled with the bond-edge type and $A_p$ be an edge in the path labeled with the bond-edge type. Without loss of generality, assume the half-edge of $A_p$ closest to the bridging vertex is labeled $a$. Let $x$ be the number of edges between the bridging vertex and the half-edge of $A_c$ labeled with $\hat{a}$, $y$ be the number of edges between the bridging vertex and the half-edge of $A_c$ labeled with $a$, and $l$ be the number of edges from the bridging vertex to $A_p$, as shown in Figure \ref{fig:tadpoleS3shortpathprooflabel}. If the half-edges of $A_c$ and $A_p$ were to re-bond with one another, a cycle of length $x+l+1$ will form, as shown in Figure \ref{fig:tadpoleS3shortpathproof}. If the resulting graph is isomorphic to $Tad_{m,n}$, then $x+l+1=x+y+1$ and $y=l$. 
 
 \begin{figure}[h!]
     \centering
     
     \begin{tikzpicture}[transform shape, scale = 0.7]
 \node[main node, fill=black, minimum size = 0.3cm] (j) at (6,-1.5) {};
 \node[main node, fill=black, minimum size = 0.3cm] (a) at (4,-1.5) {};
 \node[main node, fill=black, minimum size = 0.3cm] (b) at (4,1.5) {};
 \node[main node, fill=black, minimum size = 0.3cm] (c) at (6,1.5) {};
 \node[main node, fill=black, minimum size = 0.3cm] (d) at (8,0) {};
 \node[main node, fill=black, minimum size = 0.3cm] (e) at (10,0) {};
 \node[main node, fill=black, minimum size = 0.3cm] (f) at (12,0) {};
 \node[main node, fill=black, minimum size = 0.3cm] (g) at (14,0) {};
\path[draw]
(e) edge node [above]{$a$ \hspace{2mm} $\hat{a}$} (f)
(a) edge node [left, rotate=90, yshift=0.3cm, xshift=0.7cm]{ $\hat{a}$ \hspace{6mm} $a$} (b);

\path[draw, dashed]
(c) edge node []{} (d)
(j) edge node []{} (d)
(a) edge node []{} (j)
(b) edge node []{} (c)
(d) edge node []{} (e)
(f) edge node []{} (g);

\draw [decorate,decoration={brace,amplitude=10pt},xshift=0pt,yshift=0pt]
(4,2.5) -- (8,1.5) node [black,midway,xshift=0, yshift=1cm] 
{$y$ edges};

\draw [decorate,decoration={brace,amplitude=10pt},xshift=0pt,yshift=0pt]
(8,-1.5) -- (4,-2.5) node [black,midway,xshift=0, yshift=-1cm] 
{$x$ edges};

\draw [decorate,decoration={brace,amplitude=10pt},xshift=0pt,yshift=0pt]
(8,0.5) -- (10,0.5) node [black,midway,xshift=0, yshift=1cm] 
{$k$ edges};

\draw [decorate,decoration={brace,amplitude=10pt},xshift=0pt,yshift=0pt]
(12,0.5) -- (14,0.5) node [black,midway,xshift=0, yshift=1cm] 
{$l$ edges};

\draw []
(4,-1.5) -- (4,1.5) node [black,midway,rotate=90,xshift=0cm, yshift=1cm] 
{$A_c$};

\draw [white]
(10,0.5) -- (12,0.5) node [black,midway,xshift=0cm, yshift=0.7cm] 
{$A_p$};

\end{tikzpicture}

\caption{$Tad_{m,n}$ labeling in proof of Lemma \ref{lemma1:tadpoleS3}}
\label{fig:tadpoleS3shortpathprooflabel}

\end{figure}
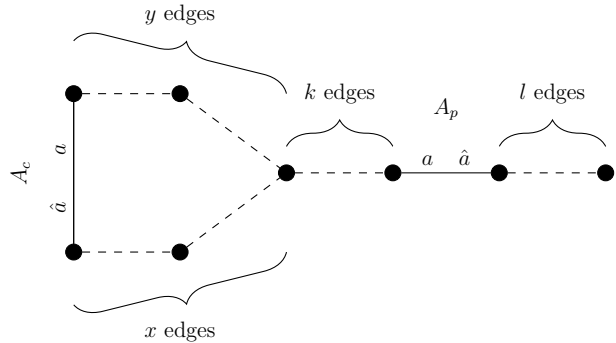

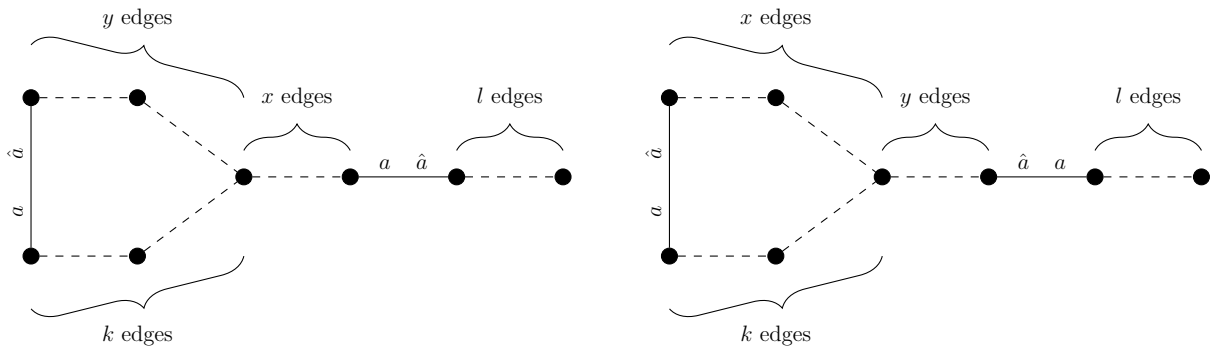
\begin{figure}
  \begin{tikzpicture}[transform shape, scale = 0.7]
 \node[main node, fill=black, minimum size = 0.3cm] (j) at (-9,-1.5) {};
 \node[main node, fill=black, minimum size = 0.3cm] (a) at (-11,-1.5) {};
 \node[main node, fill=black, minimum size = 0.3cm] (b) at (-11,1.5) {};
 \node[main node, fill=black, minimum size = 0.3cm] (c) at (-9,1.5) {};
 \node[main node, fill=black, minimum size = 0.3cm] (d) at (-7,0) {};
 \node[main node, fill=black, minimum size = 0.3cm] (e) at (-5,0) {};
 \node[main node, fill=black, minimum size = 0.3cm] (f) at (-3,0) {};
 \node[main node, fill=black, minimum size = 0.3cm] (g) at (-1,0) {};

\path[draw]
(e) edge node [above]{$a$ \hspace{2mm} $\hat{a}$} (f)
(a) edge node [left, rotate=90, yshift=0.3cm, xshift=0.7cm]{ $a$ \hspace{6mm} $\hat{a}$} (b);

\path[draw, dashed]
(c) edge node []{} (d)
(j) edge node []{} (d)
(a) edge node []{} (j)
(b) edge node []{} (c)
(d) edge node []{} (e)
(f) edge node []{} (g);

\draw [decorate,decoration={brace,amplitude=10pt},xshift=0pt,yshift=0pt]
(-11,2.5) -- (-7,1.5) node [black,midway,xshift=0, yshift=1cm] 
{$y$ edges};

\draw [decorate,decoration={brace,amplitude=10pt},xshift=0pt,yshift=0pt]
(-7,-1.5) -- (-11,-2.5) node [black,midway,xshift=0, yshift=-1cm] 
{$k$ edges};

\draw [decorate,decoration={brace,amplitude=10pt},xshift=0pt,yshift=0pt]
(-7,0.5) -- (-5,0.5) node [black,midway,xshift=0, yshift=1cm] 
{$x$ edges};

\draw [decorate,decoration={brace,amplitude=10pt},xshift=0pt,yshift=0pt]
(-3,0.5) -- (-1,0.5) node [black,midway,xshift=0, yshift=1cm] 
{$l$ edges};

 \node[main node, fill=black, minimum size = 0.3cm] (j) at (3,-1.5) {};
 \node[main node, fill=black, minimum size = 0.3cm] (a) at (1,-1.5) {};
 \node[main node, fill=black, minimum size = 0.3cm] (b) at (1,1.5) {};
 \node[main node, fill=black, minimum size = 0.3cm] (c) at (3,1.5) {};
 \node[main node, fill=black, minimum size = 0.3cm] (d) at (5,0) {};
 \node[main node, fill=black, minimum size = 0.3cm] (e) at (7,0) {};
 \node[main node, fill=black, minimum size = 0.3cm] (f) at (9,0) {};
 \node[main node, fill=black, minimum size = 0.3cm] (g) at (11,0) {};

\path[draw]
(e) edge node [above]{$\hat{a}$ \hspace{2mm} $a$} (f)
(a) edge node [left, rotate=90, yshift=0.3cm, xshift=0.7cm]{ $a$ \hspace{6mm} $\hat{a}$} (b);

\path[draw, dashed]
(c) edge node []{} (d)
(j) edge node []{} (d)
(a) edge node []{} (j)
(b) edge node []{} (c)
(d) edge node []{} (e)
(f) edge node []{} (g);

\draw [decorate,decoration={brace,amplitude=10pt},xshift=0pt,yshift=0pt]
(1,2.5) -- (5,1.5) node [black,midway,xshift=0, yshift=1cm] 
{$x$ edges};

\draw [decorate,decoration={brace,amplitude=10pt},xshift=0pt,yshift=0pt]
(5,-1.5) -- (1,-2.5) node [black,midway,xshift=0, yshift=-1cm] 
{$k$ edges};

\draw [decorate,decoration={brace,amplitude=10pt},xshift=0pt,yshift=0pt]
(5,0.5) -- (7,0.5) node [black,midway,xshift=0, yshift=1cm] 
{$y$ edges};

\draw [decorate,decoration={brace,amplitude=10pt},xshift=0pt,yshift=0pt]
(9,0.5) -- (11,0.5) node [black,midway,xshift=0, yshift=1cm] 
{$l$ edges};

\end{tikzpicture}

\caption{Non-isomorphic graphs $G_1$ (left) and $G_2$ (right) realized in proof of Lemma \ref{lemma1:tadpoleS3}}
\label{fig:tadpoleS3shortpathproof}
\end{figure}

\end{proof}

\begin{lemma}\label{lemma2:tadpoleS3}
In order to satisfy the requirements of Scenario 3 for $Tad_{m,n}$ when $n < \lceil \frac{m}{2} \rceil$, a bond-edge type used to label two edges in the cycle cannot also be used to label an edge in the path.
\end{lemma}

\begin{proof}
Let $B_{c_1}$ and $B_{c_2}$ be two edges in the cycle labeled with bond-edge type $a$, and $B_p$ be an edge in the path labeled with bond-edge type $a$. As described in the proof of Proposition \ref{prop:LollipopOddB2}, edges $B_{c_1}$ and $B_{c_2}$ must have opposite orientations in order to prevent realization of a smaller graph. Furthermore, by the argument given in the proof of Lemma \ref{lemma1:tadpoleS3}, edges $B_{c_1}$ and $B_{c_2}$ must both be the same number of edges away from the bridging vertex. If the number of edges is $k$, then edge $B_p$ must also be $k$ edges away from the bridging vertex, as shown in Figure \ref{fig:tadpoleS3shortpathlemma2}. However, in this case, the tiles of the last $n-k$ vertices in the path can bond to the complementary half-edges of $B_{c_1}$ and $B_{c_2}$ (see Figure \ref{fig:tadpoleS3shortpathlemma2nonisom}). This realizes a graph of order $3n+1$ with a central vertex and 3 extending paths. If $n < \lceil \frac{m}{2} \rceil$, $3n+1 < m+n+1$, so a smaller graph or non-isomorphic graph of equal order is realized. Note if $n \geq \lceil \frac{m}{2} \rceil$, $3n+1 \geq m+n+1$, so a graph of order greater than $m+n$ is realized. 
\end{proof}

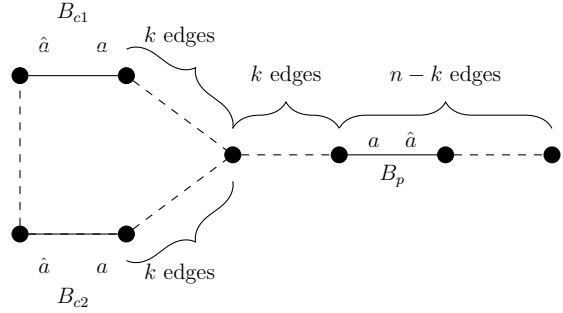
\begin{figure}[h]
\centering
  \begin{tikzpicture}[transform shape, scale = 0.7]
  
 \node[main node, fill=black, minimum size = 0.3cm] (j) at (-9,-1.5) {};
 \node[main node, fill=black, minimum size = 0.3cm] (a) at (-11,-1.5) {};
 \node[main node, fill=black, minimum size = 0.3cm] (b) at (-11,1.5) {};
 \node[main node, fill=black, minimum size = 0.3cm] (c) at (-9,1.5) {};
 \node[main node, fill=black, minimum size = 0.3cm] (d) at (-7,0) {};
 \node[main node, fill=black, minimum size = 0.3cm] (e) at (-5,0) {};
 \node[main node, fill=black, minimum size = 0.3cm] (f) at (-3,0) {};
 \node[main node, fill=black, minimum size = 0.3cm] (g) at (-1,0) {};

\path[draw]
(e) edge node [above]{$a$ \hspace{2mm} $\hat{a}$} (f)
(j) edge node [below, yshift=-0.3cm]{ $\hat{a}$ \hspace{6mm} $a$} (a)
(b) edge node [above, yshift=0.3cm]{ $\hat{a}$ \hspace{6mm} $a$} (c);

\path[draw, dashed]
(c) edge node []{} (d)
(j) edge node []{} (d)
(a) edge node []{} (j)
(d) edge node []{} (e)
(a) edge node []{} (b)
(f) edge node []{} (g);

\draw [decorate,decoration={brace,amplitude=10pt},xshift=0pt,yshift=0pt]
(-9,2) -- (-7,0.5) node [black,midway,xshift=0, yshift=1cm] 
{$k$ edges};

\draw [decorate,decoration={brace,amplitude=10pt},xshift=0pt,yshift=0pt]
(-7,-0.5) -- (-9,-2) node [black,midway,xshift=0, yshift=-1cm] 
{$k$ edges};

\draw [decorate,decoration={brace,amplitude=10pt},xshift=0pt,yshift=0pt]
(-7,0.5) -- (-5,0.5) node [black,midway,xshift=0, yshift=1cm] 
{$k$ edges};

\draw [decorate,decoration={brace,amplitude=10pt},xshift=0pt,yshift=0pt]
(-5,0.5) -- (-1,0.5) node [black,midway,xshift=0, yshift=1cm] 
{$n-k$ edges};

\draw [white]
(-11,2) -- (-9,2) node [black,midway,xshift=0cm, yshift=0.7cm] 
{$B_{c1}$};

\draw [white]
(-11,-2) -- (-9,-2) node [black,midway,xshift=0cm, yshift=-0.7cm] 
{$B_{c2}$};

\draw [white]
(-5,0.5) -- (-3,0.5) node [black,midway,xshift=0cm, yshift=-0.9cm] 
{$B_p$};

\end{tikzpicture} 

\caption{$Tad_{m,n}$ labeling in proof of Lemma \ref{lemma2:tadpoleS3}}
\label{fig:tadpoleS3shortpathlemma2}
\end{figure}

\begin{figure}[h]
    \centering
\begin{tikzpicture}[transform shape, scale = 0.7]
\centering

 \node[main node, fill=black, minimum size = 0.3cm] (h) at (3,-1.5) {};
 \node[main node, fill=black, minimum size = 0.3cm] (a) at (1,-1.5) {};
  \node[main node, fill=black, minimum size = 0.3cm] (i) at (-1,-1.5) {};
 \node[main node, fill=black, minimum size = 0.3cm] (b) at (1,1.5) {};
  \node[main node, fill=black, minimum size = 0.3cm] (j) at (-1,1.5) {};
 \node[main node, fill=black, minimum size = 0.3cm] (c) at (3,1.5) {};
 \node[main node, fill=black, minimum size = 0.3cm] (d) at (5,0) {};
 \node[main node, fill=black, minimum size = 0.3cm] (e) at (7,0) {};
 \node[main node, fill=black, minimum size = 0.3cm] (f) at (9,0) {};
 \node[main node, fill=black, minimum size = 0.3cm] (g) at (11,0) {};

\path[draw]
(h) edge node [below]{$\hat{a}$ \hspace{2mm} $a$} (a)
(c) edge node [above]{$\hat{a}$ \hspace{2mm} $a$} (b)
(e) edge node [above]{$a$ \hspace{2mm} $\hat{a}$} (f);

\path[draw, dashed]
(c) edge node []{} (d)
(h) edge node []{} (d)
(d) edge node []{} (e)
(a) edge node []{} (i)
(b) edge node []{} (j)
(f) edge node []{} (g);

\draw [decorate,decoration={brace,amplitude=10pt},xshift=0pt,yshift=0pt]
(3,2) -- (5,0.5) node [black,midway,xshift=0, yshift=1cm] 
{$k$ edges};

\draw [decorate,decoration={brace,amplitude=10pt},xshift=0pt,yshift=0pt]
(5,-0.5) -- (3,-2) node [black,midway,xshift=0, yshift=-1cm] 
{$k$ edges};

\draw [decorate,decoration={brace,amplitude=10pt},xshift=0pt,yshift=0pt]
(5,0.5) -- (7,0.5) node [black,midway,xshift=0, yshift=1cm] 
{$k$ edges};

\draw [decorate,decoration={brace,amplitude=10pt},xshift=0pt,yshift=0pt]
(7,0.5) -- (11,0.5) node [black,midway,xshift=0, yshift=1cm] 
{$n-k$ edges};

\draw [decorate,decoration={brace,amplitude=10pt},xshift=0pt,yshift=0pt]
(-1,2) -- (3,2) node [black,midway,xshift=0, yshift=1cm] 
{$n-k$ edges};

\draw [decorate,decoration={brace,amplitude=10pt},xshift=0pt,yshift=0pt]
(3,-2) -- (-1,-2) node [black,midway,xshift=0, yshift=-1cm] 
{$n-k$ edges};

\end{tikzpicture}
\caption{Non-isomorphic graph realized in proof of Lemma \ref{lemma2:tadpoleS3}}
\label{fig:tadpoleS3shortpathlemma2nonisom}
\end{figure}
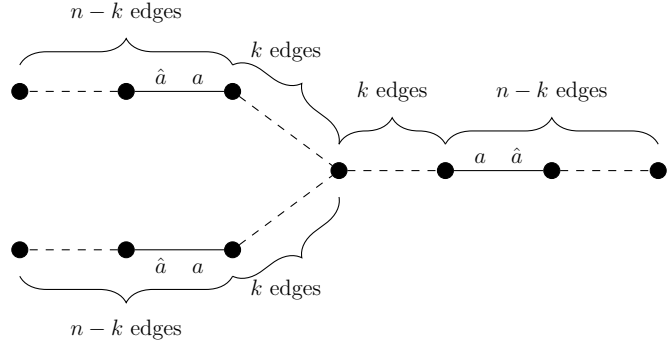

The following result again showcases the difficulties that arise in determining whether a smaller or non-isomorphic graph can be realized by a pot $P$ in which $\mathcal{S}(P)$ has one or more degrees of freedom.

\begin{proposition}
\label{prop:B3tadpoleshortpath}
$B_3(Tad_{m, n}) = \lfloor\frac{m}{2}\rfloor + n$ for $n <\lceil \frac{m}{2} \rceil$.
\end{proposition}
\begin{proof}
Proposition 13 in \cite{mintiles} shows $B_3(C_m) \geq \lceil\frac{m}{2}\rceil$. This result applies to the cycle of $Tad_{m,n}$, as a graph non-isomorphic to $Tad_{m,n}$ can be realized if fewer bond-edge types are used in the labeling of the cycle of $Tad_{m,n}$. Thus, at least $\lceil\frac{m}{2}\rceil$ bond-edge types are needed to label the cycle of $Tad_{m,n}$.  \par
By Lemmas \ref{lemma1:tadpoleS3} and \ref{lemma2:tadpoleS3}, we have two options for repeating a bond-edge type in the labeling of $Tad_{m,n}$ in Scenario 3: (1) bond-edge types may be repeated twice in the cycle but any such bond-edge type may not be repeated again in the path, (2) a bond-edge type used once in the cycle may be used to label one edge in the path at exactly the correct number of edges away from the bridging vertex. Note that when $m$ is odd, at least one bond-edge type must used to label only one edge in the cycle. Therefore, this bond-edge type can be repeated in the path at the distance specified in Lemma \ref{lemma1:tadpoleS3}. Case (1), when $m$ is even, corresponds to $\frac{m}{2}$ bond-edge types each used to label two edges in the cycle, and $n$ new bond-edge type for edges in the path. This results in $\frac{m}{2} + n$ total bond-edge types. Case (1), when $m$ is odd, corresponds to $\left\lfloor \frac{m}{2} \right\rfloor$ bond-edge types each used to label two edges in the cycle, for a total of $\left\lceil \frac{m}{2} \right\rceil$ bond-edge types used to label edges in the cycle and $n-1$ new bond-edge type for edges in the path. This results in $\left\lceil\frac{m}{2}\right\rceil + n-1 = \left\lfloor \frac{m}{2} \right\rfloor + n$ total bond-edge types. Case (2) corresponds to $m$ bond-edge types each used to label one edge in the cycle and $n$ of those bond-edge types also used to label the edges in the path subgraph. Note $\left\lfloor \frac{m}{2} \right\rfloor + n < m$ since $n < \left \lceil \frac{m}{2} \right\rceil$, so case (1) represents the minimum number of bond-edge types for $m$ even and odd. \par
We have concluded $B_3(Tad_{m,n}) \geq \left \lfloor \frac{m}{2} \right\rfloor +n$. The lower bound is achieved by the following pots for $m$ even and odd. Example labelings of $Tad_{5,2}$ and $Tad_{6,2}$ are shown in Figure \ref{fig:tadpoleS3shortpathnew}. 

\begin{equation} \label{eq:S3tadpoleshortevenpot} \begin{split} P_{(2k,n)} = \{t_1 = \{a_1^{2}, a_{k+1}\}, t_i = \{\hat{a}_{i-1}, a_i\} \text{ for } 2 \leq i \leq k, t_{k+1} = \{\hat{a}^2_{k}\},\\ t_l = \{\hat{a}_{l-1}, a_l\} \text{ for } l = k+2 \leq l \leq k+n, t_{k+n+1} = \{\hat{a}_{k+n}\}\} \end{split} \end{equation} 

\begin{equation}\label{eq:S3tadpoleshortoddpot} \begin{split} P_{(2k+1,n)} = \{t_1 = \{a_1^{2}, a_2\}, t_i = \{\hat{a}_{i-1}, a_i\} \text{ for } 2 \leq i \leq k+1, \\ t_{k+2} = \{\hat{a}_{k+1}^2\}, t_{k+3} = \{\hat{a}_1, a_{k+2}\},\\ t_l = \{\hat{a}_{l-2}, a_{l-1}\} \text{ for } k+4 \leq l \leq k+n+1, t_{k+n+2} = \{\hat{a}_{k+n}\}\} \end{split} \end{equation} 

The construction matrices and spectrums of $P_{(2k,n)}$ and $P_{(2k+1,n)}$ follow.

\begin{equation*}  M(P_{(2k,n)}) =
\begin{array}{cc} { \small \begin{array}{r} 1 \\ 2 \\ \vdots \\ \vdots\\ k \\ k+1\\ \vdots  \\ \\ k + n -1 \\ k + n \end{array} } & \begin{bmatrix}
2 & -1 & 0 & \cdots & & & & & & \cdots & 0  \\ 
0 & 1  & -1 & 0 & \cdots & & & & & \cdots & 0 \\ 
\vdots & \ddots & \ddots & \ddots & \ddots & & & & & & \vdots \\
 & & 0 & 1  & -1 & 0 & \cdots & & & &  \\
0 & \cdots & & 0 & 1  & -2 & 0 & \cdots & & & \\
1 & 0 & \cdots &  & 0 & 0  & -1 & 0 & \cdots & & \\ 
0 & \cdots & & & & 0 & 1 & \ddots & \ddots & & \vdots \\
\vdots & & &  & &  & \ddots & \ddots & -1  & 0 & 0 \\
0 & \cdots & & & & & \cdots & 0 & 1 & -1 & 0 \\
1 & \cdots & & & & & \cdots & 1 & 1 & 1 & 1 \end{bmatrix} \end{array} \end{equation*}

\begin{equation*}  M(P_{(2k+1,n)}) =
\begin{array}{cc} { \small \begin{array}{r} 1 \\ 2 \\ \vdots  \\  \\ \vdots \\ k+1 \\ k+2 \\ \vdots  \\ \\ k+n \\ \end{array} } & \begin{bmatrix}
2 & -1 & 0 & \cdots & \cdots & 0 & -1 & 0 & \cdots & & \cdots & 0 \\ 
1 & 1 & -1 & 0 & \cdots &  &  &  &  & & \cdots & 0 \\ 
0 & 0 & 1 & -1 & 0 & \cdots &  &  &  & & \cdots & 0 \\ 
\vdots & \vdots & \ddots & \ddots & \ddots & \ddots &  & & & & & \vdots \\
 & & & \ddots & 1 & -1 & 0 & \cdots & & &  &  \\
 & & & & 0  & 1 & -2 & 0 & \cdots & &  &  \\
 & & & & & 0 & 1 & -1 & 0 & \cdots &  & \\
 & & & & &  & 0 & 1 & -1 & 0 &   & \\
\vdots & \vdots &  & & & & & \ddots & \ddots & \ddots  & \ddots & \vdots \\

0 & 0 & \cdots & &   &  & & \cdots & 0 & 1 & -1  & 0  \\

1 &  1 & \cdots &  &  &  &  & & \cdots & 1 & 1 & 1  \end{bmatrix} \end{array} \end{equation*}

\begin{equation*} \begin{split} \mathcal{S}(P_{(2k,n)})=\left\{\left\langle r_1 = \frac{1}{2k+n}, r_i = \frac{2}{2k+n} \text{ for } i=2, ..., k, r_{l} = \frac{1}{2k+n} \right.\right. \\ \left. \left. \text{ for } l=\frac{2k}{2}+1,...,k+n+1 \right\rangle \right\} \end{split} \end{equation*} 

\begin{equation*} \begin{split} \mathcal{S}(P_{(2k+1,n)})=\left\{\left\langle r_1 = \frac{2}{3(2k+1)} - \frac{2n-2k-1}{3(2k+1)}r_{k+n+2}, r_2 = \frac{4}{3(2k+1)} - \frac{4n+2k+1}{3(2k+1)}r_{k+n+2}, \right. \right. \\ r_i = \frac{2}{2k+1}-\frac{2n}{2k+1}r_{k+n+2} \text{ for } 3 \leq i \leq k+1,  r_{k+n+2} = \frac{1}{2k+1}-\frac{n}{2k+1}r_{k+n+2}, \\ \left. \left.
 r_{l} = r_{k+n+2} \text{ for } \left\lceil\frac{m}{2}\right\rceil+2 \leq l \leq \left\lceil\frac{2k+1}{2}\right\rceil+n,  r_{k+n+2} \right\rangle \mid r_{k+n+2} \in \mathbb{Q}^+ \right\} \end{split} \end{equation*} 

 In $\mathcal{S}(P_{(2k,n)})$ the least common multiple of the vector component denominators in $\mathcal{S}(P_{(2k+1,n
 )})$ and $\mathcal{S}(P_{(2k,n)})$ is at least $2k+n$. Therefore, no graphs of smaller order are realized by $P_{(2k,n)}$.\par

Note that $\mathcal{S}(P_{(2k+1,n)})$ has one degree of freedom and no clear least common multiple of vector entry denominators in the spectrum, making it more difficult to show that non-existence of a solution constructing a graph smaller than $
Tad_{2k+1,n}$. \par 
In $\mathcal{S}(P_{(2k+1,n)})$, if $r_1=0$, then $r_{k+n+2} >1$, which an invalid tile proportion. Therefore, $t_1$ must be included in any complete complex. The $a_2$ arm of $t_1$ can only bond to $t_3$, and the $a_3$ arm of $t_3$ can only bond to $t_4$. This process continues until $t_{k+2}$ is bonded to the complex. In turn, this tile must bond to $t_{k+1}$, which must bond to $t_{k}$. This process continues until $t_3$ is again bonded to the complex. The complex is now of size $2k$. Tile $t_3$ can bond to tiles $t_1$ or $t_2$ via the arm $\hat{a_2}$. The remaining possibilities follow.
\begin{enumerate}
    \item Tile $t_3$ bonds to tile $t_2$, and $t_2$ bonds to the $t_1$ tile already in the complex, creating a complex of size $m$ with one extending $a_1$ arm on tile $t_1$. 
    \item Tile $t_3$ bonds to tile $t_2$ and $t_2$ bonds to a new $t_1$ tile, creating a complex of size $m+1$ with three extending $a_1$ arms and one extending $a_2$ arm on the two $t_1$ tiles. 
    \item Tile $t_3$ bonds to a new $t_1$ tile, creating a complex of size $m$ with four extending $a_1$ arms. 
\end{enumerate}
Each extending $a_1$ arm can either bond to the start of a new cycle of tiles $t_2,t_3,...,t_{k+1}$, adding $2k$ tiles to the complex. Or, each arm may bond to tile $t_{k+3}$, which must then bond to the remaining tiles in the pot $t_{k+4},...,t_{k+n+2}$, adding $n$ tiles to the complex. Since in each case 1-3 above there is a minimum of one extending $a_1$ arm, a minimum of $n$ tiles must be added to the complex. Note, $n \leq 2k$. So, in all cases, the complex is at least the target graph order, $2k+n+1$. Case 1, with tiles $t_{k+3}, t_{k+4},...,t_{k+n+2}$ bonding to the extending $a_1$ arm, represents the unique complete complex of size $2k+n+1$ that can be realized. This complex is $Tad_{2k+1,n}$. 
\end{proof}

\begin{figure}[h!]
\centering
\begin{tikzpicture}[transform shape, scale = 0.7]
 \node[main node] (i) at (-6,0) {$t_1$};
 \node[main node] (j) at (-7.38
, -1.90) {$t_3$};
 \node[main node](k) at (-9.62
,-1.17) {$t_4$};
 \node[main node] (l) at (-9.62
,1.17) {$t_3$};
 \node[main node] (m) at (-7.38, 1.90
) {$t_2$};
\node[main node] (n) at (-4,0) {$t_5$};
\node[main node] (o) at (-2,0) {$t_6$};

\path[draw,thick,color=red,->]
(i) edge node []{} (m)
(i) edge node []{} (n);

\path[draw,thick,color=blue,->]
(m) edge node []{} (l)
(i) edge node []{} (j);

\path[draw,thick,color=green,->]
(j) edge node []{} (k)
(l) edge node []{} (k);

\path[draw,thick,color=orange,->]
(n) edge node []{} (o);

 \node[main node] (a) at (5,0) {$t_1$};
 \node[main node] (b) at (4,-2) {$t_2$};
 \node[main node](c) at (2,-2) {$t_3$};
 \node[main node] (d) at (1,0) {$t_4$};
 \node[main node] (e) at (2,2) {$t_3$};
 \node[main node] (f) at (4,2) {$t_2$};
 \node[main node] (g) at (7,0) {$t_5$};
  \node[main node] (h) at (9,0) {$t_6$};
		 
\path[draw,thick,color=red,->]
(a) edge node []{} (f)
(a) edge node []{} (b);
\path[draw,thick,color=blue,->]
(f) edge node []{} (e)
(b) edge node []{} (c);
\path[draw,thick,color=green,->]
(c) edge node []{} (d)
(e) edge node []{} (d);

\path[draw,thick,color=black,->]
(g) edge node []{} (h);

\path[draw,thick,color=orange,->]
(a) edge node []{} (g);
\end{tikzpicture}

\caption{Scenario 3 labelings of $Tad_{5,2}$ and $Tad_{6,2}$} 
\label{fig:tadpoleS3shortpathnew}
\end{figure}
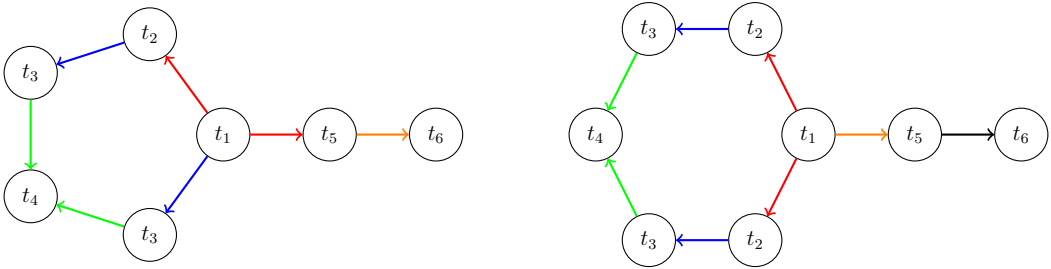

\begin{proposition}\label{prop:T3tadpoleshort} $T_3(Tad_{m, n}) = \left \lceil \frac{m}{2} \right\rceil + n + 1$ for $n < \left\lceil \frac{m}{2} \right\rceil$. \end{proposition}

\begin{proof} For $m$ even, $T_3(Tad_{m,n}) \geq \frac{m}{2} + n + 1$ by Proposition \ref{prop:B3tadpoleshortpath} and Theorem 2 of \cite{mintiles}. For $m$ odd, $T_3(Tad_{m,n}) \geq \lfloor\frac{m}{2}\rfloor + n + 1$ by Proposition \ref{prop:B3tadpoleshortpath} and Theorem 2 of \cite{mintiles}, but we claim an additional tile type is required. Proposition 13 in \cite{mintiles} shows $T_3(C_m) = \lceil\frac{m}{2}\rceil + 1$. This result applies to the cycle of $Tad_{m,n}$, as a graph non-isomorphic to $Tad_{m,n}$ can be realized if fewer tile types are used in the labeling of the cycle of $Tad_{m,n}$. Thus, at least $\lceil\frac{m}{2}\rceil + 1$ tile types are needed to label the cycle of $Tad_{m,n}$. In order to achieve the minimum of $\lceil\frac{m}{2}\rceil + 1$ tile types used to label cycle, the minimum number of bond-edge types, $\lceil\frac{m}{2}\rceil$, must be used to label the edges of the cycle. In this case, only one bond-edge type can be repeated in the path subgraph, as shown in the proof of Lemma \ref{lemma2:tadpoleS3}. Thus, no tiles used to label vertices in the cycle can be used to label vertices in the path, and $T_3(Tad_{m,n}) \geq \left \lceil \frac{m}{2} \right \rceil + 1 + n$. If more bond-edge types are used to label the cycle then the number of tile types will also necessarily increase by at least one. This process still results in a minimum of $\left \lceil \frac{m}{2} \right\rceil + 1 + n$ tile types. The pots given in (\ref{eq:S3tadpoleshortevenpot}) and (\ref{eq:S3tadpoleshortoddpot}) achieve the lower bound. \end{proof}

\begin{proposition}\label{prop:S3Tadpolelongpath}
$B_3(Tad_{m, n}) = n$ and $T_3(Tad_{m, n}) = n + 2$ for $\lceil \frac{m}{2} \rceil \leq n \leq m$.
\end{proposition}
\begin{proof}
By Lemma \ref{S3NoRepeatedBondEdge}, $B_3(Tad_{m,n}) \geq n$ and $T_3(Tad_{m,n}) \geq n$ with at least $n$ distinct tile types required to label the vertices of the path. The proof of Proposition \ref{prop:tadpoleS2mediumpath} shows that at least one additional two-armed tile type is required to label the vertices of the cycle, and the degree 3 bridging vertex requires a distinct tile type. This gives $T_3(Tad_{m,n}) \geq n+2$. The pot $P$ given in Proposition \ref{prop:tadpoleS2mediumpath} achieves the lower bounds. As shown in the proof of Proposition \ref{prop:tadpoleS2mediumpath}, $Tad_{m,n}$ is the unique graph of order $m+n$ realized by $P$.
\end{proof}

\begin{proposition}\label{prop:S3Tadpolelongestpath}
$B_3(Tad_{m, n}) = n$ and $T_3(Tad_{m, n}) = n + 1$ for $n > m$.
\end{proposition}
\begin{proof}
By Lemma \ref{S3NoRepeatedBondEdge}, $B_3(Tad_{m,n}) \geq n$ and $T_3(Tad_{m,n}) \geq n$ with at least $n$ distinct tile types required to label the vertices of the path. The degree 3 bridging vertex requires an additional distinct tile type, so $T_3(Tad_{m,n}) \geq n+1$. The pot $P$ given in Proposition \ref{prop:tadpoleS2longpath} achieves the lower bounds; $\mathcal{S}(P)$ has a unique solution. Therefore, no graphs of smaller order or of equal order but non-isomorphic to $Tad_{m,n}$ are realized by $P$.
\end{proof}

The contrast between Propositions \ref{prop:S3Tadpolelongpath} and \ref{prop:S3Tadpolelongestpath} is notable; the longer path requires fewer tile types. This is the result of a different labeling method shown in Figure \ref{fig:tadpoleS2longpath}, which only allows for fewer tile types when the path contains more vertices than the cycle.

\section{Conclusion} 
    
We have explored flexible tile based DNA self-assembly of structures resembling the lollipop and tadpole graph family. A summary of results found in three different theoretical lab scenarios is shown in Table \ref{table:Results}.

\begin{table}
\begin{tabular}{| L{0.8in} | L{0.8in} | L{2.1in} | L{1.8in} |} 
\hline
Graph Type $G$ & Scenario 1 & Scenario 2 & Scenario 3 \\ \hline

$L_{m,1}$, \newline \small{$m$ even} & $B_1(G)=1$ \newline $T_1(G)=3$ & $B_2(G)=1$ \newline $T_2(G)=3$ & $B_3(G)=m-1$ \newline $T_3(G)=m+1$ \\ \hline

$L_{m,1}$, \newline \small{$m$ odd} & $B_1(G)=1$ \newline $T_1(G)=3$ & $B_2(G)=2$ \newline $T_2(G)=4$ & $B_3(G)=m-1$ \newline $T_3(G)=m+1$ \\ \hline

$L_{m,n}$, \newline {\small $m$ even,  $1 < n \leq m$} & $B_1(G)=1$ \newline $T_1(G)=4$  & $B_2(G)=n$ \newline $T_2(G)=n+2$  &   $m+n-2 \leq B_3(G)\leq m+n-1$ \newline $T_3(G)=m+n$ \\ \hline

$L_{m,n}$, \newline {\small $m$ odd,  $1 < n \leq m$} &  $B_1(G)=1$ \newline $T_1(G)=4$  & $B_2(G) =n$ \newline $T_2(G)=n+3$ & $m+n-2 \leq B_3(G) \leq m+n-1$ \newline $T_3(G) = m+n$ \\ \hline

$L_{m,n}$, \newline {\small $m$ even,  $n > m$} & $B_1(G)=1$ \newline $T_1(G)=4$  & $\lceil \frac{n-2k+1}{2} \rceil + 2k -1 \leq B_2(G) \leq n$ \newline $\lceil \frac{n-2k+1}{2} \rceil +2k+2 \leq T_2(G) \leq n+2$  &   $m+n-2 \leq B_3(G)\leq m+n-1$ \newline $T_3(G)=m+n$ \\ \hline

$L_{m,n}$, \newline {\small $m$ odd,  $n > m$} & $B_1(G)=1$ \newline $T_1(G)=4$  & $\lceil \frac{n-2k}{2} \rceil + 2k \leq B_2(G) \leq n$ \newline $\lceil \frac{n-2k}{2} \rceil +2k+4 \leq T_2(G) \leq n+3$  &   $m+n-2 \leq B_3(G)\leq m+n-1$ \newline $T_3(G)=m+n$ \\ \hline

$Tad_{m,n}$ \newline $n < \lceil \frac{m}{2} \rceil$ & $B_1(G)=1$ \newline $T_1(G)=3$  & $B_2(G)=\lceil \frac{m}{2} \rceil$ \newline $T_2(G)=\lceil \frac{m}{2} \rceil + 2$ & $B_3(G)=\lfloor \frac{m}{2} \rfloor + n$ \newline $T_3(G)=\lceil \frac{m}{2} \rceil + n+1$ \\ \hline

$Tad_{m,n}$ \newline $\lceil \frac{m}{2} \rceil \leq n$ \newline $\leq m$ & $B_1(G)=1$ \newline $T_1(G)=3$ & $B_2(G)=n$ \newline $T_2(G)=n+2$ & $B_3(G)=n$ \newline $T_3(G)=n+2$ \\ \hline

$Tad_{m,n}$ \newline $n > m$ & $B_1(G)=1$ \newline $T_1(G)=3$ & $\lceil \frac{n-m+1}{2} \rceil + m -1 \leq B_2(G) \leq n$ \newline $\lceil \frac{n-m+1}{2} \rceil + m+1 \leq T_2(G) \leq n+1$ & $B_3(G)=n$ \newline $T_3(G)=n+1$ \\ \hline

\end{tabular}
\caption{Summary of results for lollipop and tadpole graphs} 
\label{table:Results}
\end{table}

Exact values were provided for many orders of lollipop and tadpole graphs. In even the most difficult cases, upper and lower bounds were determined. We presented three general lemmas with lower bounds for bond-edge and tile types for appending a path to a graph via a single cut-vertex. Previous results for complete graphs and cycle graphs found in \cite{mintiles} were useful in determining the minimum tile and bond-edge types within such vertex-induced subgraphs. In \cite{mintiles}, the construction matrix was immediately helpful in determining whether a smaller graph than the target graph could be realized in Scenario 2; additional methods were needed in our work to accommodate degrees of freedom in the spectrums of the pots. Lollipop graphs highlighted the significant differences that can arise by parity. Both graph families yielded several interesting examples in which minimum values are dependent on path length in surprising ways. \par 
Results from this work may be helpful in future studies involving graphs with path subgraphs appended via a single cut-vertex. This work also gives insight into the difficulties of determining optimal values for graph families that expand in order in two distinct portions of the graph. We recommend additional research to explore the relationship between a graph and its vertex-induced subgraphs. 

\subsection{Acknowledgements}
This work was supported in part by grant P20GM103499 (SC INBRE) from the National Institute of General Medical Sciences, National Institutes of Health.

\newpage

\bibliography{bibliography}

\end{document}